\documentclass[10pt, A4]{amsart} \textwidth=14.5cm \oddsidemargin=1cm \evensidemargin=1cm
\usepackage[utf8]{inputenc}

\usepackage{multirow}
\usepackage[all,cmtip]{xy}
\usepackage{caption}
\usepackage{amsmath,amssymb,amscd,mathrsfs,amscd,wasysym,marvosym,color,xcolor}
\usepackage{setspace}
\usepackage{amsthm}
\usepackage[]{xcolor}
\usepackage{mathrsfs}  
\usepackage{xypic,makecell,hhline}
\usepackage{tikz}
\usepackage{verbatim}
\usepackage{amsfonts}
\usepackage{amsxtra}
\usepackage{amssymb}
\usepackage{eucal}
\usepackage{latexsym,todonotes}
\usepackage{stmaryrd}
\usepackage{graphicx}
\usepackage{fontawesome5}
\usepackage{dynkin-diagrams}
\usepackage{tikz}
\usepackage[linktocpage=true]{hyperref}
\usepackage[capitalize]{cleveref}

\usepackage[
  backend=bibtex,
  style=numeric,   
  giveninits=true,
  doi=false,
  isbn=false,
  url=false,
  eprint=false
]{biblatex}

\DeclareFieldFormat[article,unpublished,incollection,book]{title}{#1}
\DeclareFieldFormat{pages}{#1}
\renewbibmacro{in:}{}
\usetikzlibrary{intersections}
\usetikzlibrary{shapes}
\newtheorem{theorem}{Theorem}[section]

\newtheorem{corollary}[theorem]{Corollary}

\newtheorem{definition}[theorem]{Definition}
\newtheorem{example}[theorem]{Example}

\newtheorem{lemma}[theorem]{Lemma}

\newtheorem{proposition}[theorem]{Proposition}

\numberwithin{equation}{section}

\theoremstyle{remark}
\newtheorem{remark}[theorem]{Remark}
\crefname{equation}{}{}

\newcommand{\gl}{{\mathfrak{gl}}}
\newcommand{\g}{{\mathfrak{g}}}
\newcommand{\h}{{\mathfrak{h}}}

\newcommand{\n}{{\mathfrak{n}}}

\newcommand{\I}{\mathbf I}
\def \tpsi{\psi} 

\newcommand{\new}{\varrho}

\newcommand{\II}{\mathbb{I}}
\newcommand{\half}{\frac{1}{2}}
\newcommand{\talpha}{\tilde{\alpha}}

\def \cJ{\mathcal{J} }
\def \TT{\mathbf{T}}

\def \tTD{\widetilde{T}}

\def \cK{\mathcal{K}}

\def \bcG{{\mathcal{G}}_{\bullet}}

\def \bbw{{\boldsymbol{w}}_\circ}

\def \Br{\text{Br}}
\def \bs{\mathbf{r}}
\def \ba{\mathbf{a}}
\def \reW{W^{\circ}}

\def \dm{\diamond}
\def \bIi{\I_{\bullet,i}}

\newcommand{\J}{{\bf J}}
\newcommand{\Dmn}{\mathcal{D}_{m,n}}
\newcommand{\DiAIII}{\mathcal{D}^\io}
\newcommand{\DiAI}{\mathcal{D}^{\io}_{\text{id}}}
\newcommand{\bF}{\mathbb{F}}
\newcommand{\va}{\varsigma}

\newcommand{\N}{\mathbb{N}}
\newcommand{\Z}{\mathbb{Z}}

\newcommand{\Q}{\mathbb{C}}
\newcommand{\C}{\mathbb{C}}

\newcommand{\End}{\text{End}}

\newcommand{\la}{\langle}
\newcommand{\ra}{\rangle}

\newcommand{\io}{\imath}
\newcommand{\bu}{\bullet}
\newcommand{\U}{\mathbf U}
\newcommand{\tU}{\widetilde{\U}}
\newcommand{\Ui}{{\mathbf U}^\imath}

\newcommand{\pt}{\mathsf{a}}
\newcommand{\nb}{b}
\newcommand{\up}{\Upsilon}
\newcommand{\Veven}{V_{\overline 0}}
\newcommand{\Vodd}{V_{\overline 1}}

\newcommand{\Ieven}{\I_{\overline 0}}
\newcommand{\Iodd}{\I_{\overline 1}}
\newcommand{\Phieven}{\Phi_{\overline 0}}
\newcommand{\Phiodd}{\Phi_{\overline 1}}

\newcommand{\Ub}{\U_{\bu}}
\newcommand{\qbinom}[2]{\begin{bmatrix} #1\\#2 \end{bmatrix} }

\newcommand{\wseq}{\underline{w}}
\newcommand{\relongest}{\mathbf{w}^\circ_{\I}}

\def \tTD{{T}}

\def \Id{\mathrm{Id}}

\newcommand{\zero}{{\bar{0}}}
\newcommand{\one}{{\bar{1}}}

\newcommand{\iso}{{\mathsf{iso}}}
\newcommand{\niso}{{\mathsf{n}\text{-}\mathsf{iso}}}

\newcommand{\af}{\alpha}

\newcommand{\bva}{{\boldsymbol{\varsigma}}} 

\tikzset{anchorbase/.style={>=To,baseline={([yshift=-0.5ex]current bounding box.center)}}}

\newcommand{\set}[1]{\left\{#1\right\}}

\everymath=\expandafter{\the\everymath\displaystyle} 
\allowdisplaybreaks
\title[Relative braid group symmetries of type sAIII]{Relative braid group symmetries on quantum supersymmetric pairs of type sAIII}
\author{Yaolong Shen}
\address[Y.S.]{
    Department of Mathematics and Statistics \\
    University of Ottawa \\
    Ottawa, ON, K1N 6N5, Canada
}
\urladdr{\href{https://sites.google.com/virginia.edu/yaolongshen}{sites.google.com/virginia.edu/yaolongshen}, \textrm{\textit{ORCiD}:} \href{https://orcid.org/0000-0002-8840-3394}{orcid.org/0000-0002-8840-3394}}
\email{yshen5@uottawa.ca}

\author{Weinan Zhang}
\address[W.Z.]{Department of Mathematics and New Cornerstone Science Laboratory, The University of Hong Kong, Pokfulam, Hong Kong SAR 999077, P.R.China}
\urladdr{\href{https://sites.google.com/virginia.edu/weinan/home}{sites.google.com/virginia.edu/weinan}, \textrm{\textit{ORCiD}:} \href{https://orcid.org/0000-0001-5893-1543}{orcid.org/0000-0001-5893-1543}}
\email{zhangweinan72@gmail.com}



\begin{document}

\begin{abstract} 
We introduce the relative Coxeter groupoid and construct intrinsic relative braid group symmetries for quantum supersymmetric pairs of type sAIII. These symmetries are constructed by establishing new intertwining properties of quasi $K$-matrices, which generalize the earlier non-super construction of Wang and the second author.  
We derive explicit formulas for these symmetries and prove that they satisfy the braid relations in the relative Coxeter groupoid.  
\end{abstract}

\maketitle

\setcounter{tocdepth}{1}
\tableofcontents

\section{Introduction}
\renewcommand{\thetheorem}{\Alph{theorem}}
\setcounter{theorem}{0}

\subsection{Background}
Braid group symmetries, introduced by Lusztig \cite{Lu90a,Lu90b,Lubook}, are among the most fundamental structures in the theory of Drinfeld--Jimbo quantum groups.  
They have played a central role in the construction of PBW bases and canonical bases, and have found important applications in geometric representation theory, categorification, and related areas. 

Associated with any Satake diagram \cite{Ara62}, quantum symmetric pairs, introduced by Letzter \cite{Let02} and Kolb \cite{Ko14}, provide a natural quantization of classical symmetric pairs. A quantum symmetric pair consists of a Drinfeld--Jimbo quantum group and an iquantum group, which is a coideal subalgebra of this quantum group. Over the past decade, many fundamental constructions for quantum groups have been extended to the framework of iquantum groups; we refer the reader to the exposition \cite{Wan23} for an overview of this development. 

One of the most important ingredients---the quasi $K$-matrix, which is the analogue of the quasi $R$-matrix---was introduced by Bao-Wang \cite{BW18KL} using bar involutions on iquantum groups with special parameters. A complete proof in greater generality was established in \cite{BK19}.  In \cite{WZ23}, the quasi $K$-matrix was formulated for general parameters via an intertwining property involving the anti-involution on iquantum groups (see also \cite{AV22} for a different formulation).

Using the quasi $K$-matrix, relative braid group symmetries were systematically constructed by Wang and the second author \cite{WZ23} on iquantum groups of arbitrary finite type, solving an earlier conjecture of Kolb-Pellegrini \cite{KP11}. This construction has been generalized to the quasi-split Kac-Moody setting in \cite{Zha23}, and compatible symmetries on integrable modules over quasi-split iquantum groups, together with integrality, have recently been established in \cite{WZ25}. The relative braid group symmetries have played a fundamental role in various constructions for iquantum groups, including the PBW basis and the Hall algebra realization; see \cite{LYZ25, LW22c}.

Now let $\g$ be a basic Lie superalgebra.  
A distinctive feature of a Lie superalgebra is that its Dynkin diagrams are not unique.  
Indeed, unlike the purely even case, the fundamental systems associated with \(\g\) are not necessarily conjugate under the Weyl group action, because of the presence of odd roots; see \cite{CW12}.  
In this context, one distinguishes between \emph{odd} or \emph{even} reflections, according to whether the corresponding root is odd or even.

The \emph{quantum supergroup} \(\U\), as a Drinfeld--Jimbo quantization of \(\g\), was introduced in \cite{Ya94}. The generator-relation presentation of \(\U\) depends on the Dynkin diagram of \(\g\), and hence one naturally have a family of quantum algebras for a given $\g$. Based on this family of algebras, Yamane \cite{Ya99} quantized odd reflections into algebra isomorphisms relating the different presentations of \(\U\), thus obtaining a super analogue of Lusztig's braid group symmetries. 

A notion of \emph{super Satake diagrams} \((\I=\I_\circ \cup \I_\bu,\tau)\), subject to suitable super admissible conditions, was recently introduced in \cite[Definition~2.3]{SWsuper}.  
Then the theory of quantum symmetric pairs, together with its basic structural features, including the quasi $K$-matrix was subsequently extended to the setting of quantum supersymmetric pairs associated with arbitrary Dynkin diagrams for finite-dimensional basic Lie superalgebras in \cite{SWsuper}; see also \cite{KY20,AMS25}. Since $\g$ admits non-conjugate Dynkin diagrams, isomorphic supersymmetric pairs can arise from different super Satake diagrams; see \cite[Example 2.16]{SWsuper}.  However, it is not yet clear when quantum supersymmetric pairs associated to different super Satake diagrams are isomorphic; see \cite[Remark 7.8]{SWsuper}.

\subsection{Goal and Scope}

The goal of this paper is to take a first step toward constructing relative braid group symmetries for quantum supersymmetric pairs of all basic types.  More precisely, we generalize the approach of \cite{WZ23} to the super setting and develop an intrinsic construction of relative braid group symmetries for quantum supersymmetric pairs of type sAIII, where the underlying Lie superalgebra is \(\g=\gl(m|n)\) and the underlying super Satake diagram is of the form \eqref{AIIIdiagram}.

The type sAIII case is a natural and essential testing ground for the general theory.  It is closely related to quantum supersymmetric pairs of diagonal type, for which the iquantum supergroup coincides with the usual quantum supergroup, while already exhibiting genuinely new phenomena arising from the theory of quantum supersymmetric pairs.  Thus, a complete treatment of type sAIII not only extends the construction of \cite{WZ23} beyond the classical setting but also provides the first evidence and a framework for developing relative braid group symmetries for quantum supersymmetric pairs in all basic types.

Relative braid group symmetries are expected to play a key role in future developments of the representation theory of quantum supersymmetric pairs. Rather than being confined to a single Satake diagram, they allow one to relate quantum supersymmetric pairs associated with different super Satake diagrams, and thereby provide a natural framework for comparing the corresponding module categories.

\subsection{Main results}

Let \(W\) be the Coxeter groupoid associated with \(\gl(m|n)\), as in \cref{groupoiddef}.  
In \cref{def:relgoupoid}, we introduce the relative Coxeter groupoid \(W^\circ\) of type sAIII as a subgroupoid of \(W\).  
The appearance of a groupoid, rather than a single Coxeter group, is one of the main new features in the super setting. Recall that different choices of simple systems for \(\gl(m|n)\) are related by both even and odd reflections, and the corresponding quantum supergroups may have different presentations.  
For quantum supersymmetric pairs, this phenomenon is even more delicate, since the Satake data must be transformed compatibly.

Let \(\DiAIII_\pt\) denote the collection of all sAIII super Satake diagrams of the form \cref{AIIIdiagram}. We note that \(W^\circ\) is generated by the idempotents \(e_\I\) and the simple reflections \(\bs_{i,\I}\), for all \(\I\in \DiAIII_\pt\) and \(i\in \I_\circ\). As expected, these simple reflections satisfy type B braid relations. 
Unlike the case of Lie superalgebras, the action of \(\reW\) on \(\DiAIII_\pt\) is not transitive; hence we classify its orbits in \cref{prop:orbit}.

Let $T'_{w},T_{w}$ for $w\in W$ be braid group symmetries on quantum supergroups; see \eqref{eq:TiLus} for the convention.
Let $\up_{i,\I}$ be the rank-one quasi $K$-matrix associated to $\I\in \DiAIII_\pt,i\in \I_\circ$. After establishing certain intertwining relations of the quasi $K$-matrix in \cref{sec:intertwrel}, we obtain
\begin{theorem}[Theorems~\ref{thm:ibraid}]
Let $\I,\J\in \DiAIII_\pt$ such that $\J=\bs_{i}(\I)$. 
\begin{itemize}
\item[(1)] For any $x\in \Ui(\I)$, there exists a unique element $x'\in \Ui(\J)$ such that 
\begin{equation}
x' \up_{i,\J} = \up_{i,\J} T_{\bs_{i}}'(x).
\end{equation}
Moreover, the map $x\mapsto x'$ defines an algebra isomorphism $\TT'_{i,\I}:\Ui(\I)\rightarrow \Ui(\J)$.

\item[(2)] For any $x\in \Ui(\I)$, there exists a unique element $x''\in \Ui(\J)$ such that 
\begin{equation}
x'' T_{\bs_{i}}(\up_{i,\I}^{-1}) = T_{\bs_{i}}(\up_{i,\I}^{-1}) T_{\bs_{i},\I}(x).
\end{equation}
Moreover, the map $x\mapsto x''$ defines an algebra isomorphism $\TT_{i,\I}:\Ui(\I)\rightarrow \Ui(\J)$.

\item[(3)] We have $\TT'_{i,\J} \TT_{i,\I}= \TT_{i,\J}\TT'_{i,\I}=\Id_{\I}$ 
and $\TT'_{i,\I} \TT_{i,\J}= \TT_{i,\I}\TT'_{i,\J}=\Id_{\J}$.
\end{itemize}
\end{theorem}
These symmetries \(\TT'_{i,\I}\) and \(\TT_{i,\I}\) will be referred to as \emph{relative braid group symmetries}. We also derive explicit formulas for their actions on the generators of \(\Ui(\I)\); see \cref{sec:formula} and \cref{sec:rktwo}.

In particular, these relative braid group symmetries are well defined for odd reflections as well.  
Together with our classification of the $\reW$-orbits, the new symmetries we obtain provide natural algebra isomorphisms between iquantum groups associated with different Satake diagrams. This parallels the role of odd braid group operators, which relate different presentations of quantum supergroups.

\begin{corollary}
Let $\I,\J\in \DiAIII_\pt$. If $\I,\J$ are in the same $\reW$-orbit, then $\Ui(\I)$ is isomorphic to $\Ui(\J)$.
\end{corollary}


The next result concerns the factorization of the quasi \(K\)-matrix. In the non-super setting, it was conjectured in \cite{DK19}, and later proved in \cite{WZ23}, that the quasi $K$-matrix admits a factorization into products of rank-one quasi $K$-matrices, in close analogy with the factorization property of quasi $R$-matrices. In the present super setting, we prove that the quasi \(K\)-matrix for type sAIII can be assembled similarly from the rank-one quasi \(K\)-matrices, thereby providing a canonical factorization of the quasi \(K\)-matrix. 

The factorization result is also the key input in proving the braid relations.  Indeed, the braid relations follow from the compatibility of the rank-one quasi \(K\)-matrices with the factorization of the global quasi \(K\)-matrix.  This shows that the maps \(\TT'_{i,\I}\) and \(\TT_{i,\I}\) genuinely define braid groupoid symmetries.
\begin{theorem}[{\rm Theorems~\ref{thm:property}, \ref{thm:braidreW}, and \ref{thm:braidW}}]{$\ $}

    \begin{enumerate}
        \item The quasi \(K\)-matrix associated with quantum supersymmetric pairs of type sAIII admits a canonical factorization.
        \item The symmetries \(\TT'_{i,\I}\) and \(\TT_{i,\I}\) satisfy the braid relations of the underlying relative Coxeter groupoid.
    \end{enumerate}
\end{theorem}

\subsection{Organization}
The paper is organized as follows. In \cref{sec:QG}, we review the root datum of the general linear Lie superalgebra $\g=\gl(m|n)$ and recall the basic construction of the quantum supergroup $\U=\U(\g)$. We then recall the braid group symmetries of $\U$. In \cref{sec:QSP}, we review super Satake diagrams of type sAIII and the corresponding construction of quantum supersymmetric pairs $(\U,\Ui)$. We also introduce the relative Coxeter groupoid and classify its orbits on the set of all super Satake diagrams of type sAIII. In \cref{sec:main}, we construct the relative braid group operators by establishing an intertwining property of quasi $K$-matrices, and summarize the explicit formulas for the action of these symmetries on generators. In \cref{sec:rktwo}, we derive rank \(2\) formulas for these symmetries and complete the proof of the main theorems. In \cref{sec:factorization}, we establish a factorization property of the quasi $K$-matrix, and then use it to prove that these symmetries indeed satisfy the braid relations of the relative Coxeter groupoid. 

\vspace{0.2in}

\noindent {\bf Acknowledgement: } YS was partially supported by the Fields Institute for Research in Mathematical Sciences. WZ is partially supported by the New Cornerstone Foundation through the New Cornerstone Investigator grant awarded to Xuhua He.

\section{Quantum supergroups}
\renewcommand{\thetheorem}{\thesection.\arabic{theorem}}
\setcounter{theorem}{0}

\label{sec:QG}

In this section, we set up notations for Lie superalgebras, quantum supergroups, Drinfeld doubles, and quantum supersymmetric pairs of type sAIII. When working with superspaces, we denote the parity of a homogeneous element $v$ by $\bar v \in \Z_2$. Whenever an expression involves $\bar v$, we implicitly assume that $v$ is homogeneous. For a statement $P$, we set
\[
\delta_P
:=
\begin{cases}
	1 & \text{if $P$ is true}, \\
	0 & \text{if $P$ is false}.
\end{cases}
\]
In particular, the Kronecker delta is given by $\delta_{i,j} := \delta_{i=j}$.  For any $a \in \Z_{\ge 1}$, we define
\[
    \II_a := \left \{\frac{1-a}2,\frac{3-a}{2},\ldots, \frac{a-3}{2}, \frac{a-1}2 \right\}.
\]

\subsection{Lie superalgebra of type A}
Let $V = \Veven \oplus \Vodd$ be a vector superspace over $\C$ with $\dim \Veven = m$, $\dim \Vodd = n$.  We fix a homogeneous basis
\[
    v_i,\qquad i \in \II_{m+n} = \left\{ \frac{1-m-n}{2}, \frac{3-m-n}{2}, \dotsc, \frac{m+n-1}{2} \right\},
\]
and use $|i|$ to denote the parity of $v_i$.
(Throughout, parities are always considered modulo $2$.)

The superspace $\End(V)$, equipped with the supercommutator, is a Lie superalgebra, called the general linear Lie superalgebra, which we will denote by $\mathfrak{g} = \mathfrak{gl}(V) = \gl(m|n,\C)$.  The Lie superalgebra $\mathfrak{g}$ has a homogeneous basis given by the matrix units $E_{ij}$, $i,j \in \II_{m+n}$, where the parity of $E_{ij}$ is 
\[
    \overline{E_{ij}} = |i| + |j|.
\]

Fix a Cartan subalgebra $\h$, which consists of diagonal matrices. Restricting the supertrace to $\h$, we obtain a non-degenerate symmetric bilinear form on it. Denote by $\{\epsilon_{a}\}_{a\in \II_{m+n}}$ the basis of $\h^*$ dual to the set of standard matrices $\{E_{a,a}\}_{a\in \II_{m+n}}$. Its root system $\Phi=\Phieven\oplus \Phiodd$ is given by
\begin{equation*}
\begin{aligned}
    \Phieven=&\{\epsilon_a-\epsilon_b\mid a\neq b\in \II_{m+n}, |a|=|b|\}, \\
     \Phiodd=&\{\pm(\epsilon_a-\epsilon_b)\mid |a|\neq |b|\}.
\end{aligned}
\end{equation*}

We let
\[
  \I :=\II_{m+n-1} =\left\{ 1-\frac{m+n}{2}, 1-\frac{m+n}{2}, \dotsc, \frac{m+n}{2}-1 \right\}
\]
and define the set of simple roots
\begin{equation} \label{simpleroots}
    \Pi := \left\{ \alpha_i := \epsilon_{i-\half}-\epsilon_{i+\half} : i \in \I \right\},\qquad
    \overline{\alpha_i} = |i-\half| + |i+\half|.
\end{equation}
 We denote even simple roots by $\fullmoon$ and odd simple roots by $
\otimes$. Then the corresponding Dynkin diagram is of the form
\begin{equation}
\label{eq:dynkin}
\begin{tikzpicture}[scale=1, semithick]
\node (1) [circle,draw,label=below:{$\scriptstyle{1-\frac{m+n}{2}}$},scale=0.6] at (0,0){.};
\node (2) [circle,draw,label=below:{$\scriptstyle{2-\frac{m+n}{2}}$},scale=0.6] at (1.5,0){.};
\node (3)  at (3,0) {$\cdots$} ;
\node (4) [circle,draw,scale=0.6] at (4.5,0){.};
\node (5) [circle,draw,label=below:{$\scriptstyle {\frac{m+n}{2}-1}$},scale=0.6] at (6,0){.};
\path (1) edge (2)
          (2) edge (3)
          (3) edge (4)
          (4) edge (5);
\end{tikzpicture}
\end{equation}
where $\odot=\begin{cases}
    \fullmoon & \overline{\alpha_i}=\overline{0},\\
    \otimes & \overline{\alpha_i}=\overline{1}.
\end{cases}$

Given a Dynkin diagram $\I$ of the form \eqref{eq:dynkin}, we decompose $\Pi=\Pi_\zero \sqcup \Pi_\one$ where $\Pi_s=\Pi\cap\Phi_s$ with $s \in \{\zero, \one\}$, and accordingly we have $\I=\Ieven\sqcup\Iodd$. 
For $\beta \in \Phi$, we write $|\beta|=s$ if $\beta\in \Phi_s$ with $s\in \{\zero,\one\}$. We also write $|i|:=\overline{\alpha_i}$, for any $i\in \I$. By linearity, we have a parity function $|\cdot|$ on the root lattice $X =\Z\Pi$. The Lie superalgebra $\g$ is generated by Chevalley generators $\set{e_i,f_i\mid i\in\I}$ and $\h$, and we have a  triangular decomposition $\g=\n^+\oplus \h\oplus\n^-$, where $\n^{\pm} =\oplus_{\beta\in \Phi^{\pm}}\g_\beta$. Set $\mathfrak b^+ =\h \oplus \n^+.$ 

Denote by $A=(c_{ij})_{i,j\in \I}$ the Cartan matrix for $\g$.
There exist non-zero integers $d_i$, for $i\in \I$, such that 
\begin{align}
    \label{di}
 d_i c_{ij}=d_j c_{ji}, \qquad \gcd(d_i | i\in \I)=1.
\end{align}

The set of simple coroots $\Pi^\vee=\{h_i\mid i\in \I\} \subset \h$ is given by
\[
    \alpha_j(h_i)=c_{ij},\quad \forall i,j\in \I.
\]
Let $Y=\Z \Pi^\vee$ denote the coroot lattice.The parity function $|\cdot|$ on $Y$ is given by by $|h_i|=\overline{\alpha_i}$, for all $i \in \I$, and extend by linearity as well. 

Define a symmetric bilinear form $(\cdot,\cdot):X\times X\longrightarrow \Z$ by letting
$$ (\af_i,\af_j) =d_i c_{ij},\;\;\;i,j\in \I.$$
A root $\alpha$ is called isotropic if $(\alpha, \alpha)=0$. Since we are only working with type A Lie superalgebra, all the odd simple roots are isotropic.

\subsection{Quantum supergroup of type A}
\label{QSG}
We set up notations for a quantum supergroup $\U$ of type A following \cite{Sh25,SWsuper}.

Let $q$ be an indeterminate and $\Q(q^{1/2})$ be the field of rational functions in $q^{1/2}$ with coefficients in $\Q$, the field of complex numbers. Let $\bF$ be the algebraic closure of $\Q(q^{1/2})$ and $\bF^\times:= \bF \setminus\{0\}$. Set 
 \[
 q_i:=q^{d_i},
 \]
 for $d_i$; see \eqref{di}. Denote the quantum integers and quantum binomial coefficients by
\[
[m]_i=\frac{q_i^m-q_i^{-m}}{q_i-q_i^{-1}},
\qquad
\qbinom{m}{k}_{i} =\frac{[m]_i [m-1]_i \ldots [m-k+1]_i}{[k]_i!}
\] 
for $m\in \Z, k \in \N,i\in \I$.  It will also be convenient for us to introduce the following notation. We will say that $i\neq j\in \I$ are {\em connected} if $c_{ij}\neq 0$ and write $i\sim j$. Likewise, we say that $i\neq j\in \I$ are {\em not connected} if $c_{ij}=0$ and write $i\nsim j$.

Recall $X=\Z\Pi$ and $Y=\Z\Pi^\vee$. We define $\breve{\U}_q(\g)$ to be the $\bF$-algebra generated by $E_i,F_i, K_i^{\pm 1}$, $i\in \I$, where $K_i, K_i^{-1}$ are two-sided inverse of each other, subject to the following relations: $K_i^{\pm 1}, K_j^{\pm 1}$ commute with each other, for all $i,j \in \I$,
\begin{gather}
[E_i,F_j]= E_i F_j-(-1)^{|j||i|}F_jE_i=\delta_{i,j}\frac{K_i-K_i^{-1}}{q_i-q_i^{-1}},   
\label{EF}
\\
K_i E_j  =q_i^{c_{ij}} E_j K_i,  \quad K_i F_j=q_i^{-c_{ij}} F_jK_i,
 \label{EKKF}
\end{gather}
and the quantum Serre relations,
\begin{equation}
\label{Urel}
    \begin{aligned}
        & E_i^2=F_i^2=0,\quad \text{ for }i\in \I_{\overline 1},\\
        &E_i^2E_j-[2]_iE_i E_j E_i+E_j E_i^2=0,\quad \text{ for }i\sim j, |i|=0,\\
        &F_i^2F_j-[2]_iF_i F_j F_i+F_j F_i^2=0,\quad \text{ for }i\sim j, |i|=0,\\
        &S_{|i|,|k|}(E_i,E_j,E_k)=0,
    \quad \text{ for }i\sim j\sim k,i\neq k, |j|=1,\\
        &S_{|i|,|k|}(F_i,F_j,F_k)=0,
    \quad \text{ for }i\sim j\sim k,i\neq k, |j|=1.
    \end{aligned}
\end{equation}
where $S_{t_1,t_2}(x_1,x_2,x_3)\in \bF\la x_1,x_2,x_3\ra$ is the polynomial in three non-commuting variables for $t_1,t_2\in\{\overline 0,\overline 1\}$ given by
\begin{multline}
 \label{superSerre}
    S_{t_1,t_2}(x_1,x_2,x_3)=(q+q^{-1})x_2x_3x_1x_2-[((-1)^{t_1}x_2x_3x_2x_1+(-1)^{t_1+t_1t_2}x_1x_2x_3x_2)\\
    +((-1)^{t_1t_2+t_2}x_2x_1x_2x_3+(-1)^{t_2}x_3x_2x_1x_2)].
\end{multline} 

The algebra $\breve{\U}_q(\g)$ is a Hopf superalgebra but not a Hopf algebra, and there is a simple way to modify this below if one prefers to work with Hopf algebras (cf. \cite{Ya94}). Define an algebra involution $\new$ of parity $0$ on $\breve{\U}_q(\g)$ as follows:
\begin{equation}
\label{new}
    \new(K^{\pm 1}_i)=K^{\pm 1}_i,\quad \new(E_i)=(-1)^{|i|}E_i\text{ and }
    \new(F_i)=(-1)^{|i|}F_i,\quad \forall i\in \I.
\end{equation}

Let 
\[
\U=\breve{\U}_q(\g)\oplus \breve{\U}_q(\g)\new.
\] 
Then we extend the algebra structure on $\breve{\U}_q(\g)$ to $\U$ by declaring  
\begin{equation}
\label{newrelation}
\new^2=1,\quad x\cdot \new= \new \cdot \new(x), \quad \forall x\in \U_q(\g).
\end{equation}
The comultiplication $\Delta: \U \rightarrow \U \otimes \U$ is defined as follows:
\begin{align}  \label{Delta}
\begin{split}
\Delta(E_i)  = E_i \otimes 1 + \new K_i \otimes E_i,  \quad \Delta(F_i) = \new \otimes F_i + F_i \otimes K_{i}^{-1},  \quad
 \Delta(K_{i}) = K_{i} \otimes K_{i}.
 \end{split}
\end{align}

A $\Q$-linear operator on a $\bF$-algebra  is said to be {\em anti-linear} if it sends $q^m \mapsto q^{-m}$, for $m\in\Z$.
\begin{proposition}
  \label{QGoperators}
  {\quad}
\begin{enumerate}
\item
There exists an anti-linear involution on $\U$, also denoted by $\psi$, which fixes $E_i,F_i,\new$ and swaps $K_i \leftrightarrow K_i^{-1}$, for $i\in \I $;
\item
There exists an anti-involution $\sigma$ on $\U$ such that (for all $i\in\I$)
\begin{equation*}
\sigma(E_j)=E_j,\quad\sigma(F_j)=F_j,\quad \sigma(K_{i})=(-1)^{|i|} K_{i}^{-1},\quad \sigma(\new)=\new.
\end{equation*}
\item Let $\ba=(a_i)_{i \in \I}\in (\bF^\times)^{\I}$. There exists an automorphism $\Psi_{\ba}$ on $\U$ such that (for all $i\in\I$)
\[
\Psi_{\ba}(K_i)=K_i,\quad \Psi_{\ba}(E_i)=a_i^{1/2} E_i, \quad \Psi_{\ba}(F_i)=a_i^{-1/2}F_i.
\]
\end{enumerate}
\end{proposition}

Let $\U=\bigoplus_{\nu \in X} \U_{\nu}$ be the weight decomposition of $\tU$ such that $E_i\in \U_{\alpha_i},F_i\in \U_{-\alpha_i}, \new,K_i^{\pm 1}\in \U_0$. Write $\U_\nu^+:=\U_\nu \cap \U^+.$

\begin{remark}
\label{rem:label}
In later sections, we frequently encounter \( \U \) with different presentations corresponding to different underlying Dynkin diagrams; see \cref{sec:oddrefl}. To avoid confusion, we include an additional lower subscript (when necessary) to indicate the Dynkin diagram in use.  For example, if \( \I \) is a fixed Dynkin diagram for \( \mathfrak{g} = \gl(m|n) \) as in \cref{eq:dynkin}, then we denote the corresponding quantum group by \( \U(\I) \), which is generated by
\[
E_{i,\I},\quad F_{i,\I},\quad K_{i,\I}^{\pm1},\quad \text{and} \quad \new_\I.
\]
We also have \( q_{i,\I} := q^{d_{i,\I}} \), where \( d_{i,\I} \) denotes the integers \cref{di} associated with the simple roots indexed by \( i,\I \) such that $d_{i,\I}c_{ij,\I}=(\alpha_{i,\I},\alpha_{j,\I})$ for any $i,j\in\I$.
\end{remark}

\subsection{Odd reflections}
\label{sec:oddrefl}
Let $\Dmn$ denote the set of all possible Dynkin diagrams for $\g=\gl(m|n)$ as in \eqref{eq:dynkin}. Different choices of $\I\in \Dmn$   under the Weyl group actions because of the existence of odd roots (cf. \cite[\S 1.3.6]{CW12}). 

\begin{remark}
The Dynkin diagram of $\g$ should be understood as the pair \((\I,\{d_i\}_{i\in \I})\).  
For brevity, we shall often write simply \(\I\).  
A subtlety arises when \(m=n\): in this case, the same underlying diagram \(\I\) admits two different choices of \(\{d_i\}_{i\in \I}\).  
More precisely, if \((\I,\{d_i\}_{i\in \I})\) is a Dynkin diagram of $\g$, then so is \((\I,\{-d_i\}_{i\in \I})\).  
These correspond to two different root systems, which are not conjugate under the action of the Weyl group. In what follows, we write $\I=\J$ to mean that the diagrams and $\{d_i\}$ both agree with each other.
\end{remark}

In fact, we have the following lemma,
\begin{lemma}
[\text{\cite[Proposition 2.2.1]{Ya99}, \cite[Lemma 1.26]{CW12}}]
Let $\alpha$ be an odd simple root of $\g$ in a positive system $\Phi^+$ given by the fundamental system $\Pi$. Then,
$$\Phi_{\alpha}^+:=\{-\alpha\}\cup\Phi^+\backslash \{\alpha\}$$
is a new positive system, whose corresponding fundamental system $\Pi_{\alpha}$ is given by
\begin{equation}
    \Pi_{\alpha}=\{\beta\in \Pi\mid (\beta,\alpha)=0,\beta\neq\alpha\}\cup\{\beta+\alpha\mid\beta\in \Pi,(\beta,\alpha)\neq 0\}\cup \{-\alpha\}.
\end{equation}
\end{lemma}
The operation of obtaining $\Pi_\alpha$ from $\Pi$ is denoted by $s_{\alpha}$ and referred to as an odd reflection. When $\beta\in \Pi$ is an even simple root, we abuse the notation $s_{\beta}$ to denote the even reflection associated to $\beta$.

We may identify the fundamental systems with their corresponding Dynkin diagrams. Then for any Dynkin diagram $\I$ of the form \eqref{eq:dynkin}, we let $s_{i,\I}:=s_{\alpha_i}$ for all $i\in \I$. We often drop the lower script $\I$ when it is clear from the context and only write $s_i$ for simplicity. 


\begin{lemma}[\text{\cite[Lemma 3.2]{Sh25}}]
\label{paritylemma}
If $i,j,k \in \I$ with $j\sim i$ and $j\nsim k$. Then we have
\begin{equation*}
   |\alpha_{j,s_j(\I)}|=|\alpha_{j,\I}|,\quad |\alpha_{i,s_j(\I)}|=|\alpha_{i,\I}|+|\alpha_{j,\I}|,\quad
    |\alpha_{k,s_j(\I)}|=|\alpha_{k,\I}|.
\end{equation*}
\end{lemma}
(With the convention in \cref{rem:label}, we have $|i,\I|=|\alpha_{i,\I}|$ for all $i\in \I$ and $\I\in\Dmn$.)

From \cref{paritylemma} we see that even reflections will not change the parity of any simple root (and hence preserve the Dynkin diagram) while odd reflections change the parities of the ones adjacent to it.

\begin{example}
The following is an example of an odd reflection given by $s_1$.
{\rm
\[
\xy
(0,0)*{\otimes};
(10,0)*{\otimes}**\dir{-};
(0,-4)*{\scriptstyle 1};
(10,-4)*{\scriptstyle 2};
\endxy
\overset{s_1}{\Longrightarrow}
\xy
(0,0)*{\otimes};
(10,0)*{\fullmoon}**\dir{-};
(0,-4)*{\scriptstyle 1};
(10,-4)*{\scriptstyle 2};
\endxy
\]
}
\end{example}

In \cite{HY08}, the authors introduced the {\em Coxeter groupoid} as a natural replacement for the Weyl group in the context of Lie superalgebras of basic types. We now describe how their construction applies in our setting; see also \cite[\S 3]{HY08}.

\begin{definition}[\text{\cite[Definition 1]{HY08}}]
\label{groupoiddef}
    The Coxeter groupoid $W$ of $\gl(m|n)$ is the groupoid generated by elements $e_\I$ and $s_{i,\I}$ for all $i\in\I$ and $\I\in \Dmn$ subjecting to the following relations (for all $\I\neq\J\in\Dmn$ and $i\in\I$) :
    \begin{equation}
    \label{groupoid}
        \begin{aligned}
           & e_\I^2=e_\I,\qquad e_\I e_\J=0,\\
           & e_{s_{i}(\I)}s_{i,\I}=s_{i,\I}e_\I= s_{i,\I},\qquad s_{i,s_i(\I)}s_{i,\I}=e_\I,\\
           & s_{j,s_i(\I)}s_{i,\I}=s_{i,s_j(\I)}s_{j,\I},\quad\text{ for }i\nsim j\in \I,\\
           &
           s_{i,s_js_i(\I)}s_{j,s_i(\I)}s_{i,\I}=s_{j,s_is_j(\I)}s_{i,s_j(\I)}s_{j,\I},\quad\text{ for }i\sim j\in \I.
      \end{aligned}
    \end{equation}
\end{definition}
Note that by \eqref{groupoid}, a product $w=s_{i_1,\I_1}\cdots s_{i_k,\I_k}$ is zero unless $\I_l=s_{i_{l+1},\I_{l+1}}(\I_{l+1})$ for all $1 \leq l \leq k-1$. Consequently, when $w\neq 0$, we adopt the convention of writing it more simply as $w=s_1\cdots s_{k-1}s_{k,\I_k}$. With this convention, the last two relations in \eqref{groupoid} are simplified as 
\begin{equation}
\label{braid}
    \begin{aligned}
        &s_js_{i,\I}=s_{i}s_{j,\I}\quad\text{ for }i\nsim j\in \I, \\
        &s_{i}s_{j}s_{i,\I}=s_{j}s_{i}s_{j,\I},\quad\text{ for }i\sim j\in \I,
    \end{aligned}
\end{equation}
which can be viewed as type A braid relations. For convenience, we often write $s_i(\I):=s_{i,\I}(\I)$ without ambiguity.

Define $\ell\colon W\to \N \cup \set{0,-\infty}$ to be the length function on $W$ such that $\ell(0)=-\infty$, $\ell(e_\I)=0$ for any $\I\in\Dmn$ and 
\begin{equation}
    \label{length}
    \ell(w)=\min\set{k\in \N\mid w=s_{i_1}\cdots s_{i_{k-1}}s_{{i_k},\I_k}\text{ for some }\I_k\in \Dmn}
\end{equation}
One says a product $w=s_{i_1}\cdots s_{i_{k-1}}s_{{i_k},\I_k}$ is a reduced expression of $w$ if $\ell(w)=k$. In particular, the Coxeter groupoid $W$ satisfies Matsumoto’s theorem \cite{M64}, which states that any two reduced expressions of $w$  can be transformed to each other using Coxeter relations only; see also \cite[Theorem 29]{HY08}


Therefore, for any $w\in W$ with a reduced expression $w=s_{i_1}\cdots s_{i_{k-1}}s_{{i_k},\I}$, we have $$\J=w(\I)=s_{i_1}\cdots s_{i_k}(\I)\in \Dmn.$$ Moreover, $\J$ does not depend on the choice of the reduced expression. 

\begin{lemma}[\text{\cite[Proposition 1.32]{CW12}}]
\label{DmnItoJ}
    For any $\I,\J\in \Dmn$, there exists a sequence of reflections $s_{i_1,\I_1},\ldots,s_{i_k,\I_k}$ such that $\I_k=\I$ and $s_{i_1}\cdots s_{i_k} (\I)=\J$ 
\end{lemma}

\begin{remark}
 When \( \g \) is a Lie algebra, it has a unique Dynkin diagram, and thus the associated Coxeter groupoid reduces to an actual Coxeter group; see \cite{HY08}. In this case, \( W \) has a unique longest element, allowing us to omit the subscript for the Dynkin diagram.  
\end{remark}

\begin{example}
    Consider $\g=\gl(2|2)$. We have $\I\in \mathcal{D}_{2,2}$ labelled by $\set{-1,0,1}$ for a fixed total order. Then the longest element associated to $\I$ is given by
    \[
    w_\I=s_{-1}s_0s_1s_{-1}s_0s_{-1,\I}.
    \]
\end{example}

\subsection{Braid group symmetries on $\U$}
Recall that Lusztig \cite[\S 37.1.3]{Lubook} has introduced automorphisms $T_{i,e}'$ and $T_{i,+1}''$ for $i\in \I$ and $e =\pm 1$ on quantum groups. For $i\in \Ieven$ and $e\in \{\pm 1\}$,  $T_{i,e}'$ and $T_{i,e}''$ naturally extend to automorphisms of the quantum supergroup $\U$.

For any $\I\in \Dmn$, we denote the associated quantum supergroup by $\U(\I)$ with generators $E_{i,\I},\ F_{i,\I}, K^{\pm 1}_{i,\I}$ and $\new_\I$ as in \cref{QSG}. We recall from \cite{Ya99,Sh25} the braid group operators on quantum supergroups.
\begin{proposition}[\text{\cite[Theorem 3.4]{Sh25}}]
\label{braidoperator}
Let $\I,\J\in \Dmn$ such that $\J=s_i(\I)$ for some $i\in \I$. Then there exist $\bF$-linear algebra isomorphisms $\tTD_{i,+1}'':\U(\I)\to \U(\J)$ for all $i\in\I$ satisfying
\begin{equation*}
    \tTD_{i,+1}''(E_{j,\I})=
    \begin{cases}
    -F_{i,\J}K_{i,\J},&\quad \text{if}\quad j=i, \\
    E_{i,\J}E_{j,\J}-(-1)^{|i,\J||j,\J|}q^{(\alpha_{i,\J},\alpha_{j,\J})}E_{j,\J}E_{i,\J} &\quad \text{if}\quad j \sim i,\\
    E_{j,\J} &\quad \text{if}\quad j \nsim i,
    \end{cases}
\end{equation*}
\begin{equation*}
    \tTD_{i,+1}''(F_{j,\I})=
    \begin{cases}
    -K_{i,\J}^{-1}E_{i,\J},&\quad \text{if}\quad j=i, \\
    F_{j,\J}F_{i,\J}-(-1)^{|i,\J||j,\J|}q^{-(\alpha_{i,\J},\alpha_{j,\J})}F_{i,\J}F_{j,\J} &\quad \text{if}\quad j \sim i,\\
    F_{j,\J} &\quad \text{if}\quad j \nsim i,
    \end{cases}
\end{equation*}
\begin{equation*}
    \tTD_{i,+1}''(K_{j,\I})=
    \begin{cases}
    (-1)^{|i,\J|}K_{i,\J}^{-1},&\quad \text{if}\quad j=i, \\
   (-1)^{|i,\J||j,\J|}K_{i,\J}K_{j,\J} &\quad \text{if}\quad j \sim i,\\
    K_{j,\J} &\quad \text{if}\quad j \nsim i,
    \end{cases}
\end{equation*}
\[
 \tTD_{i,+1}''(\new_\I)=\new_\J.
\]
Moreover, the $\tTD''_{i,+1}$, for $i\in \I$, satisfy the braid relations.
\end{proposition}

We can then define for any $i\in \I$ that (super analog of \cite[37.2.4]{Lubook} in $\U$)
\begin{align}
  \label{eq:sTs}
  \begin{split}
  \tTD'_{i,-1} &:=\sigma \circ \tTD''_{i,+1} \circ \sigma,
  \\
  \tTD''_{i,-e} := \tpsi \circ \tTD''_{i,+e} \circ \tpsi,
  & \qquad
  \tTD'_{i,+e} := \tpsi \circ \tTD'_{i,-e} \circ \tpsi.
  \end{split}
\end{align}
We often use the following short hand notations
\begin{align}
\label{eq:TiLus}
    T_{i}:= T''_{i,+1},\quad T'_{i}:= T'_{i,-1}.
\end{align}
Then $T_i'=T_{i}^{-1}$. Hence, we can define
\begin{align*}  
\tTD_w \equiv \tTD_{w,+1}'' :=\tTD_{i_1}\cdots \tTD_{i_r}\colon \U(\I) \to \U(w(\I)),
\end{align*}
where $w=s_{i_1} \cdots s_{i_r,\I}$ is any reduced expression of $w\in W$. In the special case where $w(\I)=\I$, $T_w$ becomes an automorphism of $\U(\I)$. Similarly, we define $T'_w$ for $w\in W$.


\section{Quantum supersymmetric pairs}
\label{sec:QSP}
In this section, we introduce the relative Coxeter groupoid associated to a super Satake diagram. We recall the definitions of the quantum supersymmetric pairs and the quasi \( K \)-matrice from \cite{Sh25} and \cite{SWsuper}. 

\subsection{Super Satake diagrams}
For a subset $\I_\bu\subset \Ieven$, denote by $\g_\bu$ the semisimple Lie subalgebra of $\g$ associated with $\I_\bu$. Denote by $W_\bu$ the Weyl group for $\g_\bu$ with the longest element $w_\bu$. Let $\tau_\bu$ denote the diagram involution of $\g_\bu$ corresponding to the element $w_\bu$. Let $\Phi_\bu \subset \Phi$ be the corresponding root system and denote $\Phi_\bu^+ =\Phi_\bu \cap \Phi^+$. We view $\g_\bu$ as a Levi subalgebra of $\g$, and therefore we have the positive coroots for $\g_\bu$ as part of coroots for $\g$. Let $\rho_\bu^\vee$ denote half the sum of the positive coroots in $\Phi_\bu$.


A super admissible pair $(\I=\I_\circ \cup\I_\bu,\tau)$ (cf. \cite[Definition 2.3]{SWsuper}), consists of subsets $\I_\bu\subset \Ieven$, $\I_\circ=\I\backslash \I_\bu$, and a (parity-preserving) Dynkin diagram involution $\tau$ such that \begin{enumerate}
    \item $\tau^2=\Id$,
    \item 
    The action of $\tau$ on $\I_\bu$ coincides with the action of $-w_\bu$,
    \item 
    If $j\in (\Ieven\cup \I_{\niso})\cap \I_\circ$ and $\tau(j)=j$, then $\alpha_j(\rho_\bu^\vee)\in \Z$, 
    \item
    If $j\in \I_{\iso}\cap \I_\circ$ and $\tau(j)=j$, then
$\alpha_j(\rho_\bu^\vee)\neq 0$
\end{enumerate}
where $\I_{\iso}$ (resp. $\I_{\niso}$) is the collection of isotropic (resp. non-isotropic) odd simple roots in $\I$; see \cite[Lemma 2.1]{SWsuper}.

\begin{remark}
Conditions (1), (2) and (3) (where $(\Ieven\cup\I_{\niso})\cap \I_\circ$ is replaced by $\I_\circ$) are exactly the conditions for admissible pairs of (non-super) Kac-Moody type  \cite[\S 2.4]{Ko14}. Condition~ (4) can be rephrased as that any isotropic simple root fixed by $\tau$ must be connected to $\I_\bu$.
\end{remark}

\begin{definition}\cite[Definition 2.8]{SWsuper}
\label{def:rank}
We define the {\em real odd (resp. even)  rank} of $(\I=\I_\circ \cup \I_\bu,\tau)$ by the number of $\tau$-orbits in $\I_\circ\cap \I_\iso$ (resp. $\I_\circ \cap(\Ieven\cup \I_\niso)$). Moreover, the real rank of $(\I=\I_\circ \cup \I_\bu,\tau)$ is defined to be the sum of its real odd and real even rank.
\end{definition}

The diagrams associated to super admissible pairs are called super Satake diagrams. We refer to \cite[Table 3]{SWsuper} for a list of super Satake diagrams of real rank one. A supersymmetric pair $(\g,\theta)$ consists of a Lie superalgebra $\g$ (of basic types) and an algebra automorphism $\theta$ of order $2$ or $4$ on $\g$; cf. \cite[Proposition 2.1]{SWsuper}. For any given super admissible pair, the authors have developed a theory of quantum supersymmetric pairs in \cite{SWsuper}. In this paper, we are mainly interested in super admissible pairs $(\g,\theta)$ of type sAIII. In this case we have $\g=\gl(m|n)$ and $\theta$ is of order $2$. We do not recall the purely even setting here, for readers of interest we refer to \cite{SW23} for a thorough treatment.

     A super Satake diagram of type sAIII with $\nb-1 =2\pt-1$ black nodes and $r$ pairs of white nodes is given as follows, where the diagram involution $\tau$ is indicated by the dashed arrows:
\begin{equation}
\label{AIIIdiagram}
\begin{tikzpicture}[yscale=.7, semithick]
\node (-4) [circle,draw,label=above:{$-\pt-r+1$},scale=0.6] at (0,0){$\cdot$};
\node (-3)  at (2,0) {$\cdots$} ;
\node (-2) [circle,draw,label=above:{$-\pt$},scale=0.6] at (4,0){$\cdot$};
\node (-1) [circle,fill,draw,label=above:{$-\pt+1$},scale=0.6] at (5,-0.5){};
\node (0)  at (5,-1.5){$\vdots$};
\node (1) [circle,fill,draw,label=below:{$\pt-1$},scale=0.6] at (5,-2.5){};
\node (2) [circle,draw,label=below:{$\pt$},scale=0.6] at (4,-3){$\cdot$};
\node (3)  at (2,-3){$\cdots$};
\node (4) [circle,draw,label=below:{$\pt+r-1$},scale=0.6] at (0,-3){$\cdot$};
\path (-4) edge (-3)
          (-3) edge (-2)
          (-2) edge (-1)
          (-1) edge (0)
          (0) edge (1)
          (1) edge (2)
          (2) edge (3)
          (3) edge (4);
\path (-4) edge[dashed,bend right,<->] (4)
    (-2) edge[dashed,bend right,<->] (2)
    (-3) edge[dashed,bend right,<->] (3);
\end{tikzpicture}
\end{equation}
 where $
\I_\bu =\{1-\pt,2-\pt,\ldots, \pt-1\},\ \I_\circ=\I\backslash \I_\bu.
$ (In case $\pt=0$, the black nodes are dropped; the nodes $\pt$ and $-\pt$ are identified and fixed by $\tau$.) We refer to $\odot$ as white dots and $\newmoon$ as black dots. Note that we have $2a+2r=m+n$. The advantage of adopting the labeling $\I = \II_{m+n-1}$ is that it allows us to describe the action of $\tau$ on simple roots as multiplication by $-1$. Since $\tau$ is parity-preserving, we must have $$|\tau(i),\I|=|-i,\I|=|i,\I| \text{ for all }i\in \I.$$
Throughout the paper, we will refer to \eqref{AIIIdiagram} for notations regarding type sAIII super Satake diagrams.

For later use, we let $\DiAIII_\pt$ denote the collection of super Satake diagrams associated with $\g = \gl(m|n)$ of the form \eqref{AIIIdiagram} with exactly $2\pt-1$ black nodes.

\begin{remark}
In \cite[\S 4.1]{Sh25}, the first author constructed super Satake diagrams of type sAIII, allowing $\I_\bu$ to include odd simple roots as well. However, the quasi $K$-matrix was only constructed for super Satake diagrams satisfying $\I_\bu\subset \Ieven$, which is a crucial requirement for defining the relative braid group symmetries later.
\end{remark}

\begin{remark}
    We note that the following diagram  {\rm $$\xy
(0,0)*{\otimes};(0,-4)*{\scriptstyle 0}
\endxy$$
with $\tau=id$ is not a super admissible pair since $\alpha_0(\rho_\bu^\vee)=0$ as $\I_\bu=\varnothing$.}
\end{remark}

\subsection{Relative Coxeter groupoids}
Let $(\I=\I_\bu \cup \I_\circ,\tau) \in \DiAIII_\pt$ be a super admissible pair of type sAIII. For any $i\in \I_\circ$, we define
\begin{align}
  \label{Iib}
 \I_{\bullet,i}:=\I_\bu\cup \{i,\tau i\}.
\end{align}
In this case, the corresponding algebra $\g_{\bu,i}$ associated with  $\I_{\bullet,i}$ has four possibilities:
\begin{enumerate}
    \item $\g_{\bu,i}=\gl(2\pt+2)$, if $i=\pm \pt$ and $|i,\I|=0$,
    \item  $\g_{\bu,i}=\gl(2|2\pt)$, if $i=\pm \pt$ and $|i,\I|=1$,
    \item  $\g_{\bu,i}=\gl(2)\times \gl(2\pt)\times \gl(2)$, if $i\neq \pm \pt$ and $|i,\I|=0$,
    \item  $\g_{\bu,i}=\gl(1|1)\times \gl(2\pt)\times \gl(1|1)$, if $i\neq \pm \pt$ and $|i,\I|=1$.
\end{enumerate}
Recall the groupoid $W$ from \cref{groupoiddef}. We define an element $\bs_{i,\I} \in W$ for  any $\I\in \DiAIII_\pt$ and $i\in \I_\circ$  as follows :
\begin{equation}
    \label{ridef}
    \bs_{i,\I}=\begin{cases}
        s_is_{-i,\I}, & i\neq \pm \pt,\\
        s_\pt\cdots s_{-\pt+1}s_{-\pt}s_{-\pt+1}\cdots s_{\pt,\I},& i=\pm \pt.
    \end{cases}
\end{equation}
By definition, we have the identifications $\I_{\bullet,i} =\I_{\bullet,\tau i}$ and $\bs_i=\bs_{\tau i} $.  For convenience, we often write $r_i(\I):=r_{i}(\I)$ without ambiguity.

\begin{lemma}
\label{bsiI}
    For any $\I\in\DiAIII_\pt$ and $i\in \I_\circ$, we have $\bs_{i}(\I)\in \DiAIII_\pt$. Moreover, we have 
    \[
    \bs_{\pt}(\I)=\I.
    \]
\end{lemma}
\begin{proof}
    It follows from a case-by-case checking directly.
\end{proof}
Consider products of the form 
\[
\vartheta=\bs_{i_1,\I_1}\cdots \bs_{i_k,\I_k}
\]
for $\I_1,\cdots,\I_k\in \DiAIII_\pt$. By \eqref{groupoid} and \cref{bsiI}, we see that $\vartheta=0$ unless $\I_l=\bs_{i,\I_{l+1}}(\I_{l+1})$ for all $1\leq l \leq k-1$. Consequently, we adopt the convention of writing it as 
\[
\vartheta=\bs_{i_1}\cdots \bs_{i_{k-1}}\bs_{i_k,\I_k}.
\]

This motivates the following construction; see also \cite{Lus76,Lus03, DK19,WZ23}.

\begin{definition}\label{def:relgoupoid}
    We define the relative Coxeter groupoid $\reW$ of a supersymmetirc pair of type sAIII to be the sub-groupoid of $W$ (see \cref{groupoiddef}) generated by elements $e_\I$ and $\bs_{i,\I}$ for all $(\I=\I_\bu\cup \I_\circ,\tau)\in\DiAIII_\pt$ and $i\in \I_\circ$.  
\end{definition}

\begin{theorem}
\label{relCgroupid}
  The relative Coxeter groupoid $\reW$ satisfy the following relations (for all $\I\neq\J\in\DiAIII_\pt$ and $i\in\I_\circ$) :
    \begin{equation}
    \label{relativegroupoid}
        \begin{aligned}
           & e_\I^2=e_\I,\qquad e_\I e_\J=0,\\
           & e_{\bs_{i}(\I)}\bs_{i,\I}=\bs_{i,\I}e_\I= \bs_{i,\I},\qquad \bs_{i,\bs_i(\I)}\bs_{i,\I}=e_\I,\\
           & \bs_{j}\bs_{i,\I}=\bs_{i}\bs_{j,\I},\quad\text{ for }i\nsim j\in \I,\\
           &
           \bs_{i}\bs_{j}\bs_{i,\I}=\bs_{j}\bs_{i}\bs_{j,\I},\quad\text{ for }i\sim j\in \I \text{ such that } i,j\neq \pm \pt,\\
           &
\bs_{\pt}\bs_{\pt+1}\bs_{\pt}\bs_{\pt+1,\I}=\bs_{\pt+1}\bs_{\pt}\bs_{\pt+1}\bs_{\pt,\I}.
      \end{aligned}
    \end{equation}
\end{theorem}

\begin{proof}
    It follows from a direct calculation as in the case of Coxeter groups.
\end{proof}

By \eqref{relativegroupoid}, we see that $\reW$ is a Coxeter groupoid itself and it satisfies the type B braid relations. Let $\ell_{\circ}$ denote the length function of the Coxeter groupoid $\reW$.
\begin{proposition}  
 \label{prop:LL}
For any $w_1, w_2 \in \reW$, the equation $\ell(w_1 w_2) =\ell(w_1) +\ell(w_2)$ holds if and only if $\ell_\circ(w_1 w_2) =\ell_\circ (w_1) +\ell_\circ (w_2)$.
\end{proposition}
\begin{proof}
    It follows from a rerun of \cite[Theorem 5.9]{Lus76}.
\end{proof}

Thus, when discussing reduced expressions of an element $w\in \reW \subset W$, we may unambiguously refer to either the Coxeter system of $\reW$ or that of $W$. Let $\relongest$ denote the longest element in $\reW$.

Unlike \cref{DmnItoJ}, the action of $\reW$ on $\DiAIII_\pt$ is not transitive. For any $\I\in \DiAIII_\pt$, recall from \eqref{AIIIdiagram} that we label the simple roots in $\I$ by $\II_{m+n-1}$. For any $i\in \I$, we suppose 
\[
\alpha_{i}=\epsilon_{i-1/2}-\epsilon_{i+1/2}.
\]
Then we get a parity sequence associated with $I$, denoted by
\[
P(\I)=\{ |\epsilon_{-\pt-r+1/2}|,|\epsilon_{-\pt-r+3/2}|,\ldots,|\epsilon_{\pt+r-1/2}|\}.
\]
By our parity assumptions in \cref{AIIIdiagram}, this parity sequence is uniquely determined by $m,n$ and its subsequence
\[
P'(\I):=\{|\epsilon_{\pt+1/2}|,|\epsilon_{\pt+3/2}|,\ldots, |\epsilon_{\pt+r-1/2}|\}
\]
since $\epsilon_{\pt-1/2},\epsilon_{\pt-3/2},\ldots, \epsilon_{-\pt+1/2}$ are all of the same parity.

Let $\text{pari}(\I)=\sum_{x\in P'(\I)}x$ denote the number of weights that have odd parity in $P'(\I)$. We say $\text{pari}(\I)$ is the parity of $\I$.

\begin{proposition}
\label{prop:orbit}
    For any $\I,\J\in \DiAIII_\pt$, there exists a sequence of reflections $r_{i_1,\I_1},\ldots,r_{i_k,\I_k}$ such that $\I_k=\I$ and $r_{i_1}\cdots r_{{i_k}}(\I)=\J$ if and only if $\text{pari}(\I)=\text{pari}(\J)$.
\end{proposition}

\begin{proof}
 Suppose there exists a sequence of reflections
\(
r_{i_1,\I_1},\ldots,r_{i_k,\I_k}
\)
such that $\I_k=\I$ and
\[
r_{i_1,\I_1}\cdots r_{i_k,\I_k}(\I)=\J.
\]
By \cref{bsiI}, each reflection $r_{a,\I_j}$ fixes the diagram $\I_j$ on which it acts, and therefore does not change $\text{pari}(\I_j)$. On the other hand, for $i\neq \pm a$, the reflection $r_{i,\I_j}$ acts by swapping adjacent weights, so it also preserves $\text{pari}(\I_j)$. Hence every reflection in the sequence preserves parity, and we conclude that
\(
\text{pari}(\I)=\text{pari}(\J).
\)

Conversely, suppose that
\(
\text{pari}(\I)=\text{pari}(\J).
\)
Then, by \cref{bsiI}, there exists a sequence
\(
s_{i_1,\I_1},\ldots,s_{i_k,\I_k}
\)
such that $\I_k$ is the subdiagram of $\I$ labeled by $\pt+1,\ldots,\pt+r-1$, and
\(
s_{i_1,\I_1}\cdots s_{i_k,\I_k}(\I_k)
\)
is the subdiagram of $\J$ labeled by $\pt+1,\ldots,\pt+r-1$. Then it follows by definition of $r_i$ that
\[
r_{i_1,\I}\cdots r_{i_k}(\I)=\J.
\]
This completes the proof.
\end{proof}

\begin{remark}
    When $(\I=\I_\circ\cup\I_\bu,\tau)$ is a symmetric pair of type AIII, then the relative Coxeter groupoid $\reW$ reduces to the {\em relative Weyl group} (cf. \cite{DK19} and \cite[\S 2.3]{WZ23}) associated with $\I$, which is a type B Weyl group itself.
\end{remark}




\begin{example}
\label{exrelative}
    Let us elaborate our construction in this subsection via an example. Consider the following type sAIII super Satake diagram whose underlying Lie superalgebra is $\mathfrak{gl}(4|2).$:
    {\rm
\begin{equation}
\label{exgl22}
\begin{tikzpicture}
    \draw[-] (-1.88, 0) -- (-1.12, 0);  
    \draw[-] (-.88, 0) -- (0, 0);   
    \draw[-] (.12, 0) -- (.88, 0);    
    \draw[-] (1.12, 0) -- (1.88, 0);    
    \node at (-2, 0) (moon1) {$\fullmoon$};
    \node at (-1, 0) (ot1) {$\otimes$};
    \node at (0, 0) (newmoon) {$\newmoon$};
    \node at (1, 0) (ot2) {$\otimes$};
    \node at (2, 0) (moon2) {$\fullmoon$};
    \node at (-2, -.4) {$\scriptstyle -2$};
    \node at (-1, -.4) {$\scriptstyle 1$};
    \node at (0, -.4) {$\scriptstyle 0$};
    \node at (1, -.4) {$\scriptstyle 1$};
    \node at (2, -.4) {$\scriptstyle 2$};
    \draw[<->, dashed, bend left=15] (-1, .4) to (1, .4); 
    \draw[<->, dashed, bend left=15] (-2, .4) to (2,.4); 
\end{tikzpicture}
\end{equation}}
where $a=1$, $\I=\I_5=\set{-2,-1,0,1,2}$ and $\I_\bu=\set{0}$. In this case the generalized Cartan matrix is 
\[A=\begin{pmatrix}
    2 & -1& 0 & 0 & 0 \\
   -1 & 0 & 1 & 0& 0\\
    0&-1 & 2 &-1&0\\
    0&0 & -1 & 0&1 \\
    0&0&0&-1& 2
\end{pmatrix}\]
with $d_{-2}=d_{-1}=d_2=1,\ d_0=d_1=-1$. Then $\DiAIII_1$ in this case is the collection of the following three super Satake diagrams
{\rm\[
\xymatrix{
\begin{tikzpicture}[scale=.8,baseline=0]
    \draw[-] (-1.88, 0) -- (-1.12, 0);  
    \draw[-] (-.88, 0) -- (0, 0);   
    \draw[-] (.12, 0) -- (.88, 0);    
    \draw[-] (1.12, 0) -- (1.88, 0);    
    \node at (-2, 0) (moon1) {$\fullmoon$};
    \node at (-1, 0) (ot1) {$\otimes$};
    \node at (0, 0) (newmoon) {$\newmoon$};
    \node at (1, 0) (ot2) {$\otimes$};
    \node at (2, 0) (moon2) {$\fullmoon$};
    \node at (-2, -.4) {$\scriptstyle -2$};
    \node at (-1, -.4) {$\scriptstyle -1$};
    \node at (0, -.4) {$\scriptstyle 0$};
    \node at (1, -.4) {$\scriptstyle 1$};
    \node at (2, -.4) {$\scriptstyle 2$};
    \draw[<->, dashed, bend left=15] (-1, .4) to (1, .4); 
    \draw[<->, dashed, bend left=15] (-2, .4) to (2,.4); 
\end{tikzpicture} &
\begin{tikzpicture}[scale=.8,baseline=0]
    \draw[-] (-1.88, 0) -- (-1.12, 0);  
    \draw[-] (-.88, 0) -- (0, 0);   
    \draw[-] (.12, 0) -- (.88, 0);    
    \draw[-] (1.12, 0) -- (1.88, 0);    
    \node at (-2, 0) (moon1) {$\otimes$};
    \node at (-1, 0) (ot1) {$\otimes$};
    \node at (0, 0) (newmoon) {$\newmoon$};
    \node at (1, 0) (ot2) {$\otimes$};
    \node at (2, 0) (moon2) {$\otimes$};
    \node at (-2, -.4) {$\scriptstyle -2$};
    \node at (-1, -.4) {$\scriptstyle -1$};
    \node at (0, -.4) {$\scriptstyle 0$};
    \node at (1, -.4) {$\scriptstyle 1$};
    \node at (2, -.4) {$\scriptstyle 2$};
    \draw[<->, dashed, bend left=15] (-1, .4) to (1, .4); 
    \draw[<->, dashed, bend left=15] (-2, .4) to (2,.4); 
\end{tikzpicture}    \ar@<.5ex>[r]^{\bs_2}  
&\begin{tikzpicture}[scale=.8,baseline=0]
    \draw[-] (-1.88, 0) -- (-1.12, 0);  
    \draw[-] (-.88, 0) -- (0, 0);   
    \draw[-] (.12, 0) -- (.88, 0);    
    \draw[-] (1.12, 0) -- (1.88, 0);    
    \node at (-2, 0) (moon1) {$\otimes$};
    \node at (-1, 0) (ot1) {$\fullmoon$};
    \node at (0, 0) (newmoon) {$\newmoon$};
    \node at (1, 0) (ot2) {$\fullmoon$};
    \node at (2, 0) (moon2) {$\otimes$};
    \node at (-2, -.4) {$\scriptstyle -2$};
    \node at (-1, -.4) {$\scriptstyle -1$};
    \node at (0, -.4) {$\scriptstyle 0$};
    \node at (1, -.4) {$\scriptstyle 1$};
    \node at (2, -.4) {$\scriptstyle 2$};
    \draw[<->, dashed, bend left=15] (-1, .4) to (1, .4); 
    \draw[<->, dashed, bend left=15] (-2, .4) to (2,.4); 
\end{tikzpicture} \ar@<.5ex>[l]^{\bs_2}  
}
\]}
with $\bs_i$ given as in \cref{ridef}. Note that $\bs_1$ sends each $\I\in \DiAIII_1$ to itself. Hence, the first and second diagram above do not lie in the same $\reW$-orbit since the first diagram has parity $0$ while the second and the third diagram have parity $1$.
\end{example}

\subsection{The iquantum supergroup}
  \label{sub:iQG}
We recall results for the quantum supersymmetric pairs of type sAIII; cf. \cite[\S 4.1]{Sh25} and \cite{SWsuper}. Recall from \cref{rem:label} the convention of labeling.

Consider a super Satake diagram $(\I=\I_\bu\cup \I_\circ, \tau) \in \DiAIII_\pt$ of the form \eqref{AIIIdiagram}. Let $w_\bu$ denote the longest element of $W_\bu$. Since $\I_\bu\subset \Ieven$, it follows from \cite[Lemma 4.2]{Sh25} that  
\begin{equation}
\label{wbI=I}
w_\bu(\I) = \I.
\end{equation}
Note that the condition \eqref{wbI=I} is crucial, as it allows us to define the associated iquantum supergroup without requiring  presentations of the quantum supergroup $\U$ associated with two distinct Dynkin diagrams; cf. \cite[\S 4.1]{Sh25}. Moreover, since every \(\I \in \DiAIII\) has the same \(\I_\bu\), we do not distinguish the longest element \( w_\bu \) across different super Satake diagrams.

Define the subalgebra $\Ub(\I)$ of $\U(\I)$ with Chevalley generators  
\[
\bcG=\bcG(\I) := \{E_{i},F_{i},K_{i}^{\pm 1},\new \mid i\in \I_\bu\}.
\]
The iquantum supergroup associated with the super Satake diagram $(\I=\I_\bu\cup \I_\circ,\tau)$ of type sAIII is the $\bF$-subalgebra of $\U(\I)$ given by  
\[
\Ui_{\bva_\I}(\I) =\langle B_{i,\I},k_{i,\I} , g_\I \mid i\in \I_\circ, g \in\bcG\rangle
\]
via the embedding $\imath:\Ui_{\bva_\I}(\I) \rightarrow \U(\I)$ with
\begin{align}
 \label{eq:Uibvs}
B_{i,\I} & \mapsto F_{i,\I} + \va_{i,\I} T_{w_\bu}(E_{\tau i,\I}) K^{-1}_{i,\I},\qquad k_{i,\I} \mapsto K_{i,\I} K_{\tau i,\I}^{-1}, \qquad \forall i\in \I_\circ.
\end{align}
where $\bva_\I=(\va_{i,\I})_{i\in\I_\circ}\in (\bF^\times)^{\I_\circ}$ is a set of parameters satisfying the following condition
    \begin{align}\label{par}
            \va_{i,\I}&=\va_{-i,\I}, \qquad \forall i\in \I .
    \end{align}
i.e., we only consider {\em balanced parameters} in this paper. Note that the condition \eqref{par} is slightly more restrictive that the condition in \cite[\S 4.2]{SWsuper} for general parameters; see \cref{rmk:par} below for the reason of using balanced parameters.

We set $\bva_\I^\dm=(\va_{i,\I})_{i\in\I_\circ}$ to be the following special parameter
\begin{align}
            \label{balance}
            \va_{i,\I}^\dm &=\va_{-i,\I}^\dm=-\sqrt{(-1)^{|i,\I|}} 
            q^{-(\alpha_{i,\I},\alpha_{i,\I}+w_\bu\alpha_{-i,\I})/2}.
    \end{align}

The parameter $\bva_\I^\dm$ \eqref{balance} is called the {\em distinguished parameter}; cf. \cite{DK19}, \cite[(2.21)]{WZ23} for its non-super counterpart.

\begin{proposition}\label{prop:par}
$\Ui_{\bva_\I}(\I)$ is isomorphic to $\Ui_{\bva_\I^\dm}(\I)$ as algebras.
\end{proposition}

\begin{proof} 
Let $\ba^\dm=(a_i^\dm)_{i\in \I}\in (\bF^\times)^{\I}$ such that
\begin{equation}
a_i^\dm=\begin{cases}
\va_{i,\I}(\va_{i,\I}^\dm)^{-1}, & \text{ if } i\in \I_\circ
\\
1, & \text{ if } i\in \I_\bu.
\end{cases}
\end{equation}
Using \cref{QGoperators}, it is straightforward to check that the automorphism $\Psi_{\ba^\dm}$ induces an isomorphism from $\Ui_{\bva_\I^\dm}(\I)$ to $\Ui_{\bva_\I}(\I)$.
\end{proof}

\begin{remark}\label{rmk:par}
It was shown in \cite[Lemma 2.5.1]{W21a} that iquantum groups with arbitrarily different parameters are all isomorphic, though not necessarily via Hopf algebra automorphisms.
Hence, it is expected that \cref{prop:par} generalizes to iquantum supergroups with general parameters. However, beyond balanced parameters, establishing \cref{prop:par} requires a complete set of defining relations for iquantum supergroups, which has not yet been explicitly formulated and is beyond the primary scope of this paper. 
\end{remark}

\begin{remark} 
For any parameter other than $\bva^\dm$, an additional rescaling of Lusztig symmetries is required (cf. \cite[\S 9.4]{WZ23} for the non-super case) in order to construct relative braid group symmetries. For the distinguished  parameter, the rescaling is trivial and hence it is the most suitable one for the construction of the relative braid group symmetries.
\end{remark}

\begin{remark}
In \cite{WZ23}, relative braid group symmetries were constructed first on universal iquantum groups, which is a subalgebra of the Drinfeld double quantum group. In that framework, the iquantum groups with parameters can be recovered through a central reduction; see \cite[\S 2.4]{WZ23}. For interested readers, it is straightforward to construct an analogue of the Drinfeld double for quantum supergroups and define a universal iquantum supergroup.
\end{remark}

In the rest of the paper, thanks to \cref{prop:par}, it suffices to only consider $\Ui_{\bva_\I^\dm}(\I)$ and we will denote 
\[
\Ui(\I):=\Ui_{\bva_\I^\dm}(\I)
\]

Let $\U^{\imath 0}(\I)$ be the subalgebra of $\Ui(\I)$ generated by $\new,k_{i,\I}, K_{j,\I}^{\pm 1}$, for $i\in \I_\circ$ and $j\in \I_\bu$. Following \cite[\S 4.4]{Sh25} and \cite[\S 4.6]{SWsuper}, we recall a monomial basis for $\Ui$. Set $B_{j,\I} =F_{j,\I}$ for $j\in \I_\bu$. Given a multi-index $J=(j_1,j_2,\dots,j_n) \in \I^n$, define  
\[
F_{J,\I} := F_{j_1,\I}F_{j_2,\I}\cdots F_{j_n,\I}, \quad 
B_{J,\I} := B_{j_1,\I}B_{j_2,\I}\cdots B_{j_n,\I}.
\]
Let $\cJ$ be a fixed subset of $\bigcup_{n\geq 0} \I^n$ such that $\{F_{J,\I} \mid J\in \cJ\}$ forms a basis of $\U(\I)$ as a $\U^+(\I) \U^0(\I)$-module.

\begin{proposition}[\text{cf. \cite[\S 4.6]{SWsuper}}]
\label{Iwasawabasis}
The set $\{B_{J,\I} \mid J\in \cJ \}$ forms a basis of the left (or right) $\Ub^+(\I) \U^{\imath 0}(\I)$-module $\Ui(\I)$.
\end{proposition}

\subsection{Quasi \(K\)-matrix}
Fix $\I\in\DiAIII_\pt$, where $\I=(\I_\bu\cup \I_\circ,\tau)$ is of the form \cref{AIIIdiagram}. The diagram involution $\tau$ can be naturally extended to an algebra isomorphism of $\U(\I)$ such that for any $i\in\I$ (cf. \cite[(3.11)]{SWsuper})
\begin{equation}
\label{tau}
    \tau(E_{i,\I})=(-1)^{|i,\I|}E_{\tau i,\I},\quad \tau(F_{i,\I})=F_{\tau i,\I},\quad \tau(K_{i,\I})=K_{\tau i,\I},\quad \tau(\new_{{i,\I}})=\new_{i,\I}.
\end{equation}
Recall the anti-involution $\sigma$ from \cref{QGoperators}.

\begin{theorem}[\text{\cite[Theorem 6.18]{Sh25},\cite[Theorem 6.9]{SWsuper}}]
\label{Kmatrix}
There is a unique element 
\[
\up_\I =\sum_{\mu \in X^+_{\overline{0}}} \up_{\mu,\I}
\qquad (\text{with } \up_{0,\I}=1, \up_{\mu,\I} \in \U^+_\mu(\I)), 
\]
where $X^+_{\overline 0} =\big\{\beta \in X \mid \beta \in \oplus_{i\in \I} \N \alpha_i, |\beta|=0 \big\}$
such that 
\begin{align}
B_{\tau i,\I}\up_{\I} &=\up_{\I} \sigma \tau (B_{i,\I}),  \qquad \text{ for } i\in \I_\circ,
\label{intertwine}
\\
x_\I \up_{\I} &=\up_{\I} x_\I,  \qquad\qquad \text{ for } x_\I \in \U^{\imath 0}(\I) \U_\bu(\I).
\label{xup=upx}
\end{align}
Moreover, $\up_{\mu,\I}=0$ unless $\mu \in X^+_{\overline{0}}$ and $w_\bu \tau (\mu) =\mu$.
\end{theorem}
For $i\in \I_\circ$, the rank-one quasi $K$-matrix
\[
\up_{i,\I} =1+\sum_{m\geq 1} \up_{i,\I}[m],\qquad \up_{i,\I}[m]\in \U^+_{m(\alpha_{i,\I}+w_\bu\alpha_{\tau i,\I})}(\I)\subset \U^+_{\bIi}(\I) 
\]
is defined to be the quasi $K$-matrix associated to the rank-one super Satake subdiagram $(\I_\bu\cup \{i,\tau i\}, \tau)$. By definition, we immediately have $\up_{i,\I} =\up_{\tau i,\I}$. 


\begin{lemma}\label{lem:iparity}
Let $\I,\J\in \DiAIII_\pt$ such that $\J=\bs_{i}(\I)$. Then we have $(\alpha_{i,\I},\alpha_{i,\I})=(\alpha_{i,\J},\alpha_{i,\J})$ and $(\alpha_{i,\I},w_\bu\alpha_{-i,\I})=(\alpha_{i,\J},w_\bu \alpha_{-i,\J})$.
\end{lemma}

\begin{proof}
Note that $|i,\I|=|i,\J|$. If $|i,\I|=\overline{0}$, then we have $\J=\I$. If $|i,\I|=\overline{1}$, then we have $(\alpha_{i,\I},\alpha_{i,\I})=(\alpha_{i,\J},\alpha_{i,\J})=0$. For the second identity, it suffices to show it for $i=\pt$. This case can be checked by a direct computation.
\end{proof}

For later use, we also record the following results.
\begin{lemma}[\textnormal{cf. \cite[Proposition~4.13]{BW18QSP}}]
\label{lem:blackfix}
Let $j\in \I_\bu$. Then
\[
T_j(\up)=\up,
\qquad
T_j(\up_{i,\I})=\up_{i,\I}
\quad\text{for all } i\in \I_\bu.
\]
\end{lemma}
Although \cite[Proposition~4.13]{BW18QSP} is stated for purely even setting, the same argument applies in the present setting.

\begin{proposition}[\textnormal{\cite[Theorem 5.2]{SWsuper}}]
For any $j\in \I_\bu$, braid group symmetries $T_j,T_j'$ restrict to algebra automorphisms on $\Ui$.
\end{proposition}

\subsection{Some intertwining relations}
\label{sec:intertwrel}
Let $\I\in \DiAIII_\pt$ and $i\in \I_\circ$. Set $\J=\bs_{i}(\I)$. Following \cite[Theorem 6.1]{WZ23}, we define
\begin{align}\label{def:up''}
\up''_{i,\J} := T_{\bs_i} \big( \up_{i,\I}^{-1} \big).
\end{align}
The element \( \up''_{i,\J} \) will play a central role in establishing the relative braid group symmetries in the remaining part of this section.

\begin{lemma}\label{lem:rktwo}
Let $\I,\J\in \DiAIII_\pt$ such that $\J=\bs_{i}(\I)$.
For any $j\in \I_\circ$ such that $j\neq i,\tau i$, we have
\begin{align}\label{eq:rktwo1}
F_{j,\I} \up_{i,\I}&=\up_{i,\I}F_{j,\I},
\\\label{eq:rktwo2}
T_{w_\bu} (E_{\tau j,\J}) K_{j,\J}^{-1}\up''_{i,\J}&=\up''_{i,\J}T_{w_\bu} (E_{\tau j,\J}) K_{j,\J}^{-1}.
\end{align}
\end{lemma}

\begin{proof}
The proof is the same as its non-super counterpart in \cite[Proof of Lemma 5.1]{WZ23}.
Note that by definition the weight of any component of $\up_{i,\I}$ is of parity $\overline{0}$.
\end{proof}

\begin{lemma}
\label{lem:rkzero}
Let $\I,\J\in \DiAIII_\pt$ such that $\J=\bs_{i}(\I)$. For any $x_\I \in \U^{\imath0}(\I)\U_\bu(\I)$, the elements $T'_{\bs_i}(x_\I)$, $T_{\bs_i}(x_\I)$ commute with $\up_{i,\J}$ and $\up_{i,\J}''$.
\end{lemma}

\begin{proof}
For any $x_\I \in \U^{\imath0}(\I)\U_\bu(\I)$, note that $T_{\bs_i}(x_\I)\in \U^{\imath0}(\J)\U_\bu(\J)$. Then the desired statement follows by \cref{Kmatrix}.
\end{proof}

Set $\cK_{i,\J}:=K_{i,\J} T_{w_\bu}(K_{-i,\J}^{-1})$ for $i\in \J_\circ$. Note that $\cK_{i,\J}\in \Ui(\J)$. For $i\in \J_\circ$, set $\tau_i=1$ if $i\neq \pm \pt$; $\tau_i=\tau$ if $i=\pm \pt$.

\begin{proposition}
\label{prop:qsrkone}
Let $\I,\J\in \DiAIII_\pt$ such that $\J=\bs_{i}(\I)$. 
We have 
\begin{align}
\label{eq:rkone2}
\Big( (-1)^{|i,\J|}o_{i,\J} q^{(\alpha_{i,\J},w_\bu\alpha_{-i,\J}-\alpha_{i,\J})/2} T_{w_\bu}^{-2}(B_{-\tau_i(i),\J}) \cK_{-\tau_i(i),\J} \Big)\up''_{i,\J}&=\up''_{i,\J} T_{\bs_{i},\I}(B_{i,\I}),
\\\label{eq:rkone3}
\Big( (-1)^{|i,\J|}o_{i,\J} q^{(\alpha_{i,\J},w_\bu\alpha_{-i,\J}-\alpha_{i,\J})/2} T_{w_\bu}^{-2}(B_{\tau_i(i),\J}) \cK_{\tau_i(i),\J} \Big)\up''_{i,\J}&=\up''_{i,\J} T_{\bs_{i},\I}(B_{-i,\I}).
\end{align}
\end{proposition}

\begin{proof}
Note that $|i,\I|=|i,\J|$ and then $o_{i,\I}=o_{i,\J}$. By \cref{braidoperator} and \cref{Kmatrix}, we have
\begin{align*}
\up''_{i,\J} T_{\bs_{i},\I}(B_{i,\I}) \big(\up''_{i,\J}\big)^{-1}
&=T_{\bs_{i},\I}\big(\up_{i,\I}^{-1}B_{i,\I}\up_{i,\I}\big)
=T_{\bs_{i},\I}\big(\sigma\tau (B_{-i,\I})\big)
\\
&=T_{\bs_{i},\I}\Big(F_{i,\I}-o_{i,\I}q^{(\alpha_{i,\I},w_\bu\alpha_{-i,\I}-\alpha_{i,\I})/2} T_{w_\bu}^{-1}(E_{-i,\I}) K_{i,\I}\Big)
\\
&=-q^{-(\alpha_{i,\J},\alpha_{i,\J})}T_{w_\bu}^{-1}\big(E_{\tau_i(i),\J} K_{\tau_i(i),\J}^{-1}\big)
\\
&\quad +(-1)^{|i,\J|}o_{i,\I} q^{(\alpha_{i,\I},w_\bu\alpha_{-i,\I}-\alpha_{i,\I})/2} T_{w_\bu}^{-2}(F_{-\tau_i(i),\J} K_{-\tau_i(i),\J}) T_{w_\bu}^{-1}(K_{\tau_i(i),\J}^{-1})
\\
&=(-1)^{|i,\J|}o_{i,\J} q^{(\alpha_{i,\J},w_\bu\alpha_{-i,\J}-\alpha_{i,\J})/2} T_{w_\bu}^{-2}(B_{-\tau_i(i),\J}) 
K_{-\tau_i(i),\J} T_{w_\bu}(K_{\tau_i(i),\J}^{-1}),
\end{align*}
where we used \cref{lem:iparity} in the last equality. This proves \eqref{eq:rkone2}. The second identity \eqref{eq:rkone3} can be proved similarly.
\end{proof}

\begin{proposition}
\label{prop:qsrkone'}
Let $\I,\J\in \DiAIII_\pt$ such that $\J=\bs_{i}(\I)$. 
We have 
\begin{align}\label{eq:rkone4}
\Big(o_{i,\J} q^{(\alpha_{i,\J},\alpha_{i,\J}-w_\bu\alpha_{-i,\J})/2} T_{w_\bu}^2(B_{-\tau_i(i),\J}) \cK_{-\tau_i(i),\J}^{-1}\Big)\up_{i,\J}&= \up_{i,\J} T'_{\bs_i,\I}(B_{i,\I}),
\\\label{eq:rkone5}
\Big(o_{i,\J} q^{(\alpha_{i,\J},\alpha_{i,\J}-w_\bu\alpha_{-i,\J})/2} T_{w_\bu}^2(B_{\tau_i(i),\J}) \cK_{\tau_i(i),\J}^{-1}\Big)\up_{i,\J}&= \up_{i,\J} T'_{\bs_i,\I}(B_{-i,\I}).
\end{align}
\end{proposition}

\begin{proof}
Recall that $\bs_i w_\bu$ is the longest element in the subdiagram $\{i,\tau i\}\cup \I_\bu$. By \cref{lem:iparity}, we have
\begin{align*}
T'_{\bs_i,\I}(B_{i,\I})&=T_{w_\bu} T'_{\bs_i w_\bu,\I}(B_{i,\I})
\\
&= -(-1)^{|i,\J|} T_{w_\bu}(E_{\tau_i(i),\J}K_{\tau_i(i),\J})
\\
&\quad +o_{i,\J} q^{-(\alpha_{i,\J},\alpha_{i,\J}+w_\bu\alpha_{-i,\J})/2}
T_{w_\bu}^2(K_{-\tau_i(i),\J}^{-1}F_{-\tau_i(i),\J}) T_{w_\bu}(K_{\tau_i(i),\J})
\\
&=o_{i,\J} q^{(\alpha_{i,\J},\alpha_{i,\J}-w_\bu\alpha_{-i,\J})/2}
T_{w_\bu}^2( F_{-\tau_i(i),\J}) \cK_{-\tau_i(i),\J}^{-1}
\\
&\quad -(-1)^{|i,\J|} T_{w_\bu}(E_{\tau_i(i),\J}K_{\tau_i(i),\J})
\\
&=o_{i,\J} q^{(\alpha_{i,\J},\alpha_{i,\J}-w_\bu\alpha_{-i,\J})/2}
T_{w_\bu}^2\left( \sigma \tau (B_{-\tau_i(i),\J})\right) \cK_{-\tau_i(i),\J}^{-1}.
\end{align*}
Note that $\up_{i,\J}$ is fixed by $T_{w_\bu}$. Hence, by \cref{Kmatrix}, we have 
\begin{align*}
\up_{i,\J}T'_{\bs_i,\I}(B_{i,\I}) \Big(\up_{i,\J}\Big)^{-1}
=o_{i,\J} q^{(\alpha_{i,\J},\alpha_{i,\J}-w_\bu\alpha_{-i,\J})/2}
T_{w_\bu}^2\big(  B_{-\tau_i(i),\J}\big) \cK_{-\tau_i(i),\J}^{-1}.
\end{align*}
This proves the first identity. The second identity \eqref{eq:rkone5} can be proved similarly.
\end{proof}

\section{Main results}\label{sec:main}
In this section, using the quasi \( K \)-matrix, we construct explicitly the relative braid group operators \( \TT_{i,\I} \) for \(i\in \I_\circ\) on \( \Ui(\I) \) via an intertwining property analogous to \cite[Theorems 4.7 and 6.1]{WZ23}.

\subsection{New symmetries on iquantum supergroups}

For $\I\in \DiAIII_\pt$, let $\Id_{\I}$ denote the identity map on $\Ui(\I)$. Recall from \eqref{eq:TiLus} that $T_{w,\I},T'_{w,\I}$ denote braid group symmetries on $\U$.

\begin{theorem}
\label{thm:ibraid}
Let $\I,\J\in \DiAIII_\pt$ such that $\J=\bs_{i}(\I)$. 
\begin{itemize}
\item[(1)] For any $x\in \Ui(\I)$, there exists a unique element $x'\in \Ui(\J)$ such that 
\begin{equation}\label{eq:upTT-1}
x' \up_{i,\J} = \up_{i,\J} T_{\bs_{i},\I}'(x).
\end{equation}
Moreover, the map $x\mapsto x'$ defines an algebra isomorphism $\TT'_{i,\I}:\Ui(\I)\rightarrow \Ui(\J)$.

\item[(2)] For any $x\in \Ui(\I)$, there exists a unique element $x''\in \Ui(\J)$ such that 
\begin{equation}\label{eq:upTT}
x'' T_{\bs_{i},\I}(\up_{i,\I}^{-1}) = T_{\bs_{i},\I}(\up_{i,\I}^{-1}) T_{\bs_{i},\I}(x).
\end{equation}
Moreover, the map $x\mapsto x''$ defines an algebra isomorphism $\TT_{i,\I}:\Ui(\I)\rightarrow \Ui(\J)$.

\item[(3)] We have $\TT'_{i,\J} \TT_{i,\I}= \TT_{i,\J}\TT'_{i,\I}=\Id_{\I}$ 
and $\TT'_{i,\I} \TT_{i,\J}= \TT_{i,\I}\TT'_{i,\J}=\Id_{\J}$.

\end{itemize}
\end{theorem}

\begin{proof}
A complete proof requires developments in Sections~\ref{sec:main}-\ref{sec:rktwo} and an outline of the proof is given below. The strategy in the proof is similar to the proof of \cite[Theorem 4.7]{WZ23} (see also \cite[Theorem 3.1]{Zha23}).

Firstly, since $\up_{i,\I}$ is invertible, the uniqueness of $x'$ and $x''$ is clear. 

We show the existence of $x',x''$ for all generators of $\Ui(\I)$. The existence of $x',x''$ for $x\in \U^{\imath0}(\I)\Ub(\I)$ is established in \cref{lem:rkzero}.
We next consider the case $x=B_{j,\I},j\in \I_\circ$. If $j=i,\tau i$, then the existence of $x',x''$ for $x=B_{j,\I}$ is established in Propositions~\ref{prop:qsrkone}-\ref{prop:qsrkone'}. Hence, it remains to consider $j\neq i,\tau i$. If $j\not \sim i$ and $j\not \sim \tau i$, then $(B_{j,\I})'=(B_{j,\I})''=B_{j,\I}$ clearly satisfy \eqref{eq:upTT-1}-\eqref{eq:upTT}.
Thus we assume $j \sim i$ or $j\sim \tau i$ in the remaining cases; these cases are separated into 3 classes depending on the $\tau$-orbit of $i$:

\begin{itemize}
\item[(i)] If $i\neq \pm \pt$, then the existence of $(B_{j,\I})''$ for $j \sim i$ or $j\sim \tau i$ is established in Propositions~\ref{prop:Tirktwo1}-\ref{prop:Tirktwo2}; the existence of $(B_{j,\I})'$ for $j \sim i$ or $j\sim \tau i$ is established in Propositions~\ref{prop:Tirktwo3}-\ref{prop:Tirktwo4}. 

\item[(ii)] If $i=0$, then the existence of $x',x''$ for $x=B_j,j\sim 0$ is established in \cref{prop:TT0}. 

\item[(iii)] If $i=\pm \pt,a\neq 0$, then the existence of $(B_{j,\I})''$ for $j\neq i,\tau i$ is established in \cref{prop:qs}; the existence of $(B_{j,\I})'$ for $j\neq i,\tau i$ is established in \cref{prop:qsqs}. 
\end{itemize}

Hence, we have showed the existence of $x',x''$ for all generators of $\Ui(\I)$.

Assume that for $x,y\in \Ui(\I)$, there exist $x'',y''\in \Ui(\J)$ satisfying \eqref{eq:upTT}. Then $(xy)'':=x'' y''$ satisfies \eqref{eq:upTT} for $xy$. This shows that $x''$ exist for any $x\in \Ui(\I)$. Similar for $x'$. Moreover, by the above construction, both $x\mapsto x''$ and  $x\mapsto x'$ are well-defined algebra homomorphisms from $\Ui(\I)$ to $\Ui(\J)$, which we denote by $\TT_{i,\I}$ and $\TT'_{i,\I}$ respectively. By definition, they satisfy
\begin{align}\label{eq:intertw}
\begin{split}
\TT'_{i,\I}(x) \up_{i,\J} &= \up_{i,\J} T_{\bs_{i},\I}'(x),
\\
\TT_{i,\I}(x)  \,T_{\bs_{i},\I}(\up_{i,\I}^{-1})&= T_{\bs_{i},\I}(\up_{i,\I}^{-1}) \,T_{\bs_{i},\I}(x).
\end{split}
\end{align}

Finally, we deduce that identities in part (3) hold. Indeed, these identities directly follow from the intertwining relation \eqref{eq:intertw}. In particular, (3) implies that $\TT_{i,\I}$ and $\TT'_{i,\I}$ are algebra isomorphisms.
\end{proof}

\begin{corollary}\label{cor:iso}
Suppose that $\I,\J\in \DiAIII_a$ lying in the same $\reW$-orbits. Then $\Ui(\I)$ is isomorphic to $\Ui(\J)$ as algebras.
\end{corollary}

\subsection{Formulas of symmetries $\TT'_{i}, \TT_{i}$}
\label{sec:formula}
In the type sAIII setting (see \eqref{AIIIdiagram}), the white vertices $\I_\circ$ can be separated into the following 3 classes.
\begin{itemize}
\item[(i)] $i\neq \pm \pt$.

\item[(ii)] $i=\pt=0$.

\item[(iii)] $i=\pm \pt,\pt\neq 0$.
\end{itemize}
In this subsection, we will summarize formulas for $\TT'_{i,\I},\TT_{i,\I}$ in all these 3 classes.

In what follows, we denote
\begin{equation}
    \label{qijJ}
    \begin{split}
    q_{i,j,\J}&=q^{d_{i,\J}a_{ij,\J}}=q^{(\alpha_{i,\J},\alpha_{j,\J})},
    \qquad
    [a_{ij,\J}]_i=\frac{q^{a_{ij,\J}}_{i,\J}-q_{i,\J}^{-a_{ij,\J}}}{q_{i,\J}-q_{i,\J}^{-1}},
    \qquad \forall i,j\in \J.
    \end{split}
\end{equation}
We also define
\[
    [x,y]_\pt := xy-(-1)^{|x||y|} a yx, \quad \forall a \in \bF,
    \qquad \text{and} \qquad [x,y] := [x,y]_1.
\]

\begin{theorem}\label{thm:formula}
 Let $\I,\J\in \DiAIII_\pt$ such that $\J=\bs_{i}(\I)$. 
\begin{itemize}
\item [(0)] The actions of $\TT_{i,\I}$ for $i\in\I_\circ$ on $\U_\bu(\I)\U^{\imath0}$ are given by
\begin{align}
\label{eq:TTxk}
&\TT_{i,\I}(x_\I)=x_\J,\quad \forall x\in  \U_\bu(\I);
\qquad
\TT_{i,\I}(k_j)=T_{\bs_i}(k_j),\quad \forall j\in \I_\circ.
\end{align}
\item [(i)] Let $i\neq a$. The actions of $\TT_{i,\I}$ on $B_{j,\I}$ are given by 
\begin{align}\label{eq:TiBj}
&\TT_{i,\I}(B_{j,\I})=
\begin{cases}
(-1)^{|i,\J|}o_{i,\J} q^{-(\alpha_{i,\J},\alpha_{i,\J})/2} B_{\tau i,\J} k_{\tau i,\J}, & \text{ if } j=i, 
\\
(-1)^{|i,\J|}o_{i,\J} q^{-(\alpha_{\tau i,\J},\alpha_{\tau i,\J})/2} B_{i,\J} k_{ i,\J}, & \text{ if } j={\tau i}, 
\\
[B_{j,\J},B_{i,\J}]_{q_{i,j,\J}^{-1}},& \text{ if }  j\sim i,j\not\sim \tau i,
\\
[B_{j,\J},B_{\tau i,\J}]_{q_{\tau i,j,\J}^{-1}},& \text{ if } j\sim \tau i,j\not\sim i,
\vspace{0.1in}
\\
\makecell{\big[[B_{j,\J},B_{i,\J}]_{q_{i,j,\J}^{-1}},B_{\tau i,\J}\big]_{q_{i,j,\J}^{-1}} \qquad \\
-o_{i,\J} (-1)^{|i,\J|} [a_{ij,\J}]_i q^{-(\alpha_{i,\J},\alpha_{i,\J})/2-(\alpha_{i,\J},\alpha_{j,\J})}B_{j,\J} k_{i,\J},}
& \text{ if } \tau i\sim j\sim i.
\end{cases}
\end{align}

\item[(ii)] The action of $\TT_{0,\I}$ on $B_{j,\I}$ is given by 
\begin{align}\label{eq:T0Bj}
\TT_{0,\I}(B_{j,\I})=
\begin{cases}
B_{j,\I}, & \text{ if } j=0 \text{ or } j \not\sim 0,
\\
[B_{j,\I},B_{i,\I}]_{q_{i,j,\I}^{-1}}, & \text{ if } j\sim 0.
\end{cases}
\end{align}

\item[(iii)] The actions of $\TT_{\pt,\I}$ on $B_{j,\I}$ are given by 
\begin{align}\label{eq:TaBj}
&\TT_{\pt,\I}(B_{j,\I})=
\begin{cases}
(-1)^{|\pt,\J|}o_{\pt,\J} q^{(\alpha_{\pt,\I},w_\bu\alpha_{-\pt,\I}-\alpha_{\pt,\I})/2} T_{w_\bu}^{-2}(B_{\pt,\J}) \cK_{\pt,\J}, & \text{ if } j=\pt, 
\\
(-1)^{|\pt,\J|}o_{\pt,\J} q^{(\alpha_{\pt,\I},w_\bu\alpha_{-\pt,\I}-\alpha_{\pt,\I})/2} T_{w_\bu}^{-2}(B_{-\pt,\J}) \cK_{-\pt,\J}, & \text{ if } j=-\pt, 
\vspace{0.1in}
\\
\vspace{0.1in}
\makecell{\big[[B_{\pt+1,\J},B_{\pt,\J}]_{q_{\pt,\pt+1,\J}^{-1}}, 
T_{w_\bu}^{-1}(B_{-\pt,\J})\big]_{q_{\pt,\pt+1,s_\pt(\I)}^{-1}}
\\
-o_{\pt,\J} (-1)^{|\pt,\J|} q^{-(\alpha_{\pt,\J},\alpha_{\pt,\J}+w_\bu \alpha_{-\pt,\J})/2} [c_{\pt,\pt+1,\J}]_\pt \cK_{\pt,\J} B_{\pt+1,\J}},
& \text{ if } j=\pt+1, 
\\
\vspace{0.1in}
\makecell{\big[[B_{-\pt-1,\J},B_{-\pt,\J}]_{q_{\pt,\pt+1,\J}^{-1}}, 
T_{w_\bu}^{-1}(B_{\pt,\J})\big]_{q_{\pt,\pt+1,s_\pt(\I)}^{-1}}
\\
-o_{\pt,\J} q^{-(\alpha_{\pt,\J},\alpha_{\pt,\J}+w_\bu \alpha_{-\pt,\J})/2} [c_{\pt,\pt+1,\J}]_\pt\cK_{-\pt,\J} B_{-\pt-1,\J}},
& \text{ if } j=-\pt-1, 
\\
B_{j,\J}, & \text{ otherwise.}
\end{cases}
\end{align}
\end{itemize}
\end{theorem}

\begin{proof}
The formula \eqref{eq:TiBj} of the action of $\TT_{i,\I},i\neq \pm \pt$ follows from \cref{prop:qsrkone} and Propositions~\ref{prop:Tirktwo1}-\ref{prop:Tirktwo2}. The formula \eqref{eq:T0Bj} follows from \cref{prop:qsrkone} and \cref{prop:TT0}. The formula \eqref{eq:TaBj} follows from \cref{prop:qsrkone} and \cref{prop:qs}.
\end{proof}

\begin{theorem}\label{thm:formula2}
 Let $\I,\J\in \DiAIII_\pt$ such that $\J=\bs_{i}(\I)$. 
\begin{itemize}
\item [(0)] The actions of $\TT'_{i,\I}$ for $i\in\I_\circ$ on $\U_\bu(\I)\U^{\imath0}$ are given by
\begin{align}
\label{eq:TTxk2}
&\TT'_{i,\I}(x_\I)=x_\J,\quad \forall x\in  \U_\bu(\I);
\qquad
\TT'_{i,\I}(k_j)=T'_{\bs_i}(k_j),\quad \forall j\in \I_\circ.
\end{align}
\item [(i)] Let $i\neq a$. The actions of $\TT'_{i,\I}$ on $B_{j,\I}$ are given by 
\begin{align}\label{eq:TiBj2}
&\TT'_{i,\I}(B_{j,\I})=
\begin{cases}
o_{i,\J} q^{(\alpha_{i,\J},\alpha_{i,\J})/2} B_{\tau i,\J} k_{\tau i,\J}^{-1}, 
& \text{ if } j=i, 
\\
o_{i,\J} q^{(\alpha_{\tau i,\J},\alpha_{\tau i,\J})/2} B_{i,\J} k_{i,\J}^{-1}, 
& \text{ if } j={\tau i}, 
\\
[B_{i,\J},B_{j,\J}]_{q_{i,j,\J}^{-1}},& \text{ if }  j\sim i,j\not\sim \tau i,
\\
[B_{\tau i,\J},B_{j,\J}]_{q_{\tau i,j,\J}^{-1}},& \text{ if } j\sim \tau i,j\not\sim i,
\vspace{0.1in}
\\
\makecell{
 \big[B_{\tau i,\J},[B_{i,\J},B_{j,\J}]_{q_{i,j,\J}^{-1}}\big]_{q_{i,j,\J}^{-1}} 
 \\
-o_{i,\J} [a_{ij,\J}]_i q^{-(\alpha_{i,\J},\alpha_{i,\J})/2-(\alpha_{i,\J},\alpha_{j,\J})}B_{j,\J} k_{\tau i,\J} ,}
& \text{ if } \tau i\sim j\sim i.
\end{cases}
\end{align}

\item[(ii)] The action of $\TT'_{0,\I}$ on $B_{j,\I}$ is given by 
\begin{align}\label{eq:T0Bj2}
\TT'_{0,\I}(B_{j,\I})=
\begin{cases}
B_{j,\I}, & \text{ if } j=0 \text{ or } j \not\sim 0,
\\
[B_{i,\I},B_{j,\I}]_{q_{i,j,\I}^{-1}}, & \text{ if } j\sim 0.
\end{cases}
\end{align}

\item[(iii)] The actions of $\TT'_{\pt,\I}$ on $B_{j,\I}$ are given by 
\begin{align}\label{eq:TaBj2}
&\TT'_{\pt,\I}(B_{j,\I})=
\begin{cases}
o_{\pt,\J} q^{(\alpha_{\pt,\I},\alpha_{\pt,\I}-w_\bu\alpha_{-\pt,\I})/2} T_{w_\bu}^{2}(B_{\pt,\J}) \cK_{\pt,\J}^{-1}, & \text{ if } j=\pt, 
\\
o_{\pt,\J} q^{(\alpha_{\pt,\I},\alpha_{\pt,\I}-w_\bu\alpha_{-\pt,\I})/2} T_{w_\bu}^{2}(B_{-\pt,\J}) \cK_{-\pt,\J}^{-1}, & \text{ if } j=-\pt, 
\vspace{0.1in}
\\
\makecell{ \big[T_{w_\bu}(B_{-\pt,\J}), [B_{\pt,\J},B_{\pt+1,\J}]_{q_{\pt,\pt+1,\J}^{-1}}\big]_{q_{\pt,\pt+1,s_\pt(\I)}^{-1}}
\\
-o_{\pt,\J}  q^{-(\alpha_{\pt,\J},\alpha_{\pt,\J}+w_\bu \alpha_{-\pt,\J})/2} [c_{\pt,\pt+1,\J}]_\pt  B_{\pt+1,\J}\cK_{\pt,\J}^{-1},} 
& \text{ if } j=\pt+1, 
\vspace{0.1in}
\\
\vspace{0.1in}
\makecell{ \big[T_{w_\bu}(B_{\pt,\J}), [B_{-\pt,\J},B_{-\pt-1,\J}]_{q_{\pt,\pt+1,\J}^{-1}}\big]_{q_{\pt,\pt+1,s_\pt(\I)}^{-1}}
\\
-o_{\pt,\J}  (-1)^{|\pt,\J|} q^{-(\alpha_{\pt,\J},\alpha_{\pt,\J}+w_\bu \alpha_{-\pt,\J})/2} [c_{\pt,\pt+1,\J}]_\pt  B_{-\pt-1,\J}\cK_{-\pt,\J}^{-1},} 
& \text{ if } j=-\pt-1, 
\\
B_{j,\J}, & \text{ otherwise. }
\end{cases}
\end{align}
\end{itemize}
\end{theorem}

\begin{proof}
The formula \eqref{eq:TiBj2} of the action of $\TT'_{i,\I},i\neq \pm \pt$ follows from \cref{prop:qsrkone'} and Propositions~\ref{prop:Tirktwo3}-\ref{prop:Tirktwo4}. The formula \eqref{eq:T0Bj2} follows from \cref{prop:qsrkone'} and \cref{prop:TT0}. The formula \eqref{eq:TaBj2} follows from \cref{prop:qsrkone'} and \cref{prop:qsqs}.
\end{proof}

\begin{remark}
On the level of generators, the two relative braid group symmetries
$\TT_i$ and $\TT'_i$ are related by the assignment
\[
B_i\mapsto B_i,\qquad 
k_i\mapsto (-1)^{|i,\I|}k_i^{-1},\qquad
x\mapsto \sigma(x),\quad x\in \U_\bu(\I).
\]

At present, however, we cannot verify directly that this assignment extends to a well-defined anti-involution of $\Ui(\I)$ since we do not have a presentation of $\Ui(\I)$ by generators and relations in the super setting yet. Also, unlike the non-super case \cite[Proposition 3.12]{WZ23}, it seems hard to establish this anti-involution using the quasi $K$-matrix. Note that by \cref{QGoperators}, $\sigma(k_i)=k_i^{-1}$ does not have an additional sign since $k_i=K_i K_{\tau i}^{-1}$.

If such $\sigma^\imath$ exists, then one has
$\TT'_{i}= \sigma^\imath \circ \TT_{i} \circ \sigma^\imath.$
\end{remark}

\section{Verification of formulas of symmetries $\TT'_{i},\TT_{i}$}\label{sec:rktwo}
In this section, we show the existence of $x',x''$ in Theorem~\ref{thm:ibraid} for $x=B_{j,\I},j\neq i,\tau i$ and complete the proof of Theorem~\ref{thm:ibraid}. In the meantime, we obtain nontrivial rank-two formulas of symmetries $\TT'_{i},\TT_{i}$ and complete the proof of \cref{thm:formula} and \cref{thm:formula2}.

\subsection{Symmetries $\TT'_{i},\TT_{i},i\neq \pm \pt$}
In this subsection, we fix $i\in [a+1,a+r-1]$. In this case, $\bs_{i,\I}=s_{i}s_{\tau i,\I}$. Set $o_{i,\I}=o_{\tau i,\I}=\sqrt{(-1)^{|i,\I|}}$. By \eqref{balance}, $B_{i,\I}=F_{i,\I}-o_{i,\I} q^{-(\alpha_{i,\I},\alpha_{i,\I})/2} E_{\tau i,\I} K_{i,\I}^{-1}$. By \cref{QGoperators} and \cref{tau}, we have
\[
\sigma\tau(B_{\tau i,\I})=F_{i,\I}-o_{i,\I} q^{-(\alpha_{i,\I},\alpha_{i,\I})/2}  E_{\tau i,\I} K_{i,\I}.
\]

Let \( \I, \J \in \DiAIII \) such that \( \J = \bs_{i,\I}(\I) \). Then we have \( \I_\bu = \J_\bu \), and consequently \( \Ub(\I)=\Ub(\J) \). For any \( x_\I \in \Ub(\I) \), we denote by \( x_\J \) the identical element in \( \Ub(\J) \), with the notation highlighting its presentation relative to \( \J \).

Recall $q_{i,j,\J}$ from \eqref{qijJ}.
\begin{proposition}
\label{prop:Tirktwo1}
Let $\I,\J\in \DiAIII_\pt$ such that $\J=\bs_{i}(\I)$. 
For any $j\in \I_\circ$ such that $j\sim i$ and $j\not\sim \tau i$, we have 

\begin{itemize}
\item[(1)] $(\alpha_{i,\J},\alpha_{i,\J}+2\alpha_{j,\J})
=(\alpha_{j,\I},\alpha_{j,\I})-(\alpha_{j,\J},\alpha_{j,\J});$

\item[(2)]
$[B_{j,\J},B_{i,\J}]_{q_{i,j,\J}^{-1}} \up''_{i,\J} = \up''_{i,\J} T_{\bs_{i},\I}(B_{j,\I}).$
\end{itemize}
\end{proposition}
 


\begin{proof}
The identity (1) can be verified case by case. Note that, by Lemma~\ref{paritylemma}, if $|i,\J|=\bar{1}$, then $|j,\I|$ and $|j,\J|$ must have different parities.

We show the relation (2). It suffices to show the following two identities:
\begin{align}
\label{eq:diag1}
[F_{j,\J},B_{i,\J}]_{q_{i,j,\J}^{-1}} \up''_{i,\J} &= \up''_{i,\J} T_{\bs_{i},\I}(F_{j,\I}),
\\\label{eq:diag2}
 \big[T_{w_\bu}(E_{\tau j,\J})
K_{j,\J}^{-1},B_{i,\J}\big]_{q_{i,j,\J}^{-1}} \up''_{i,\J}
& = 
o_{i,\I}  q^{(\alpha_{j,\J},\alpha_{j,\J})/2-(\alpha_{j,\I},\alpha_{j,\I})/2}  \up''_{i,\J} T_{\bs_{i},\I}\Big(T_{w_\bu}(E_{\tau j,\I})K_{j,\I}^{-1}\Big).
\end{align}
(Note that $o_{i,\I}$ appears in \eqref{eq:diag2} since $|i,\I| +|j,\I|=|j,\J|$.)

Let us prove \eqref{eq:diag1}. By a direct computation, we have 
\begin{align*}
[F_{j,\J},B_{i,\J}]_{q_{i,j,\J}^{-1}}&=[F_{j,\J},F_{i,\J}]_{q_{i,j,\J}^{-1}}
-o_{i,\J} q^{-(\alpha_{i,\J},\alpha_{i,\J})/2} [F_{j,\J},E_{\tau i,\J} K_{i,\J}^{-1}]_{q_{i,j,\J}^{-1}}
\\
&=[F_{j,\J},F_{i,\J}]_{q_{i,j,\J}^{-1}}
-o_{i,\J} q^{-(\alpha_{i,\J},\alpha_{i,\J})/2} [F_{j,\J},E_{\tau i,\J}] K_{i,\J}^{-1}
\\
&= [F_{j,\J},F_{i,\J}]_{q_{i,j,\J}^{-1}}=T_{\bs_{i},\I}(F_{j,\I}).
\end{align*}
Thus, \eqref{eq:diag1} is equivalent to \eqref{eq:rktwo1}, and then \eqref{eq:diag1} follows.

Let us prove \eqref{eq:diag2}. By Theorem~\ref{Kmatrix}, $k_{\tau i,\J}$ commutes with $\up''_{i,\J}$. By Proposition~\ref{prop:qsrkone}, we have
\begin{align}\label{eq:diag0}
(\up''_{i,\J})^{-1} B_{i,\J}  \up''_{i,\J} =  o_{i,\I} q^{(\alpha_{i,\J},\alpha_{i,\J})/2} T_{\bs_{i},\I}(B_{\tau i,\I})k_{\tau i,\J} .
\end{align}
Using this formula and Lemma~\ref{lem:rktwo}, we obtain
\begin{align*}
&\quad (\up''_{i,\J})^{-1} \big[T_{w_\bu}(E_{\tau j,\J})K_{j,\J}^{-1},B_{i,\J}\big]_{q_{i,j,\J}^{-1}}  \up''_{i,\J} 
\\
&= o_{i,\I} q^{(\alpha_{i,\J},\alpha_{i,\J})/2} \big[T_{w_\bu}(E_{\tau j,\J})K_{j,\J}^{-1},T_{\bs_{i},\I}(B_{\tau i,\I})k_{\tau i,\J}\big]_{q_{i,j,\J}^{-1}}
\\
&= o_{i,\I} q^{(\alpha_{i,\J},\alpha_{i,\J})/2} \big[T_{w_\bu}(E_{\tau j,\J})K_{j,\J}^{-1},T_{\bs_{i},\I}(F_{\tau i,\I})k_{\tau i,\J}\big]_{q_{i,j,\J}^{-1}}
\\
&= -o_{i,\I} q^{-(\alpha_{i,\J},\alpha_{i,\J} )/2} \big[T_{w_\bu}(E_{\tau j,\J}), E_{\tau i,\J}\big]_{q_{i,j,\J}^{-1}}K_{j,\J}^{-1} K_{\tau i,\J}^{-1}k_{\tau i,\J}
\\
&=(-1)^{|i,\J||j,\J|}o_{i,\I}q^{-(\alpha_{i,\J},\alpha_{i,\J}+2\alpha_{j,\J})/2} \big[E_{\tau i,\I}, T_{w_\bu}(E_{\tau j,\J})\big]_{q_{i,j,\J}}K_{j,\J}^{-1} K_{i,\J}^{-1}
\\
&=o_{i,\I} q^{(\alpha_{j,\J},\alpha_{j,\J})/2-(\alpha_{j,\I},\alpha_{j,\I})/2}
T_{\bs_{i},\I}\Big(T_{w_\bu}(E_{\tau j,\J})K_{j,\J}^{-1}\Big)
\end{align*}
where we used the identity (1) in the last equality.

The desired \eqref{eq:diag2} follows from the above computation directly.
\end{proof}

\begin{proposition}
\label{prop:Tirktwo2}
Let $\I,\J\in \DiAIII_\pt$ such that $\J=\bs_{i}(\I)$. For any $j\in \I_\circ$ such that $\tau i\sim j\sim i$, we have 
\begin{align*} 
&\up''_{i,\J} T_{\bs_{i},\I}(B_{j,\I})\big(\up''_{i,\J}\big)^{-1}
\\
=& \big[[B_{j,\J},B_{i,\J}]_{q_{i,j,\J}^{-1}},B_{\tau i,\J}\big]_{q_{i,j,\J}^{-1}} 
-o_{i,\J}(-1)^{|i,\J|}[a_{ij,\J}]_i q^{-(\alpha_{i,\J},\alpha_{i,\J})/2-(\alpha_{i,\J},\alpha_{j,\J})}B_{j,\J} k_{i,\J} .
\end{align*}
\end{proposition}

\begin{proof}
Note that the case $\tau i\sim j\sim i$ occurs only if $\I_\bu=\emptyset,\tau j=j$. In this case, we have $d_i a_{ij}=d_{\tau i}a_{\tau i,j}$ and $|j|=\bar{0}$. It suffices to prove the following two relations
\begin{align}\label{eq:diag3}
& \up''_{i,\J} T_{\bs_{i},\I}(F_{j,\I})\big(\up''_{i,\J}\big)^{-1}
\\\notag
=& \big[[F_{j,\J},B_{i,\J}]_{q_{i,j,\J}^{-1}},B_{\tau i,\J}\big]_{q_{i,j,\J}^{-1}} 
-o_{i,\J}(-1)^{|i,\J|} [a_{ij,\J}]_i q^{-(\alpha_{i,\J},\alpha_{i,\J})/2-(\alpha_{i,\J},\alpha_{j,\J})} F_{j,\J} k_{i,\J} ,
\\\label{eq:diag4}
 &\up''_{i,\J} T_{\bs_{i},\I}\big(E_{j,\I}K_{j,\I}^{-1}\big)\big(\up''_{i,\J}\big)^{-1}
 \\\notag
 =& q^{(\alpha_{j,\I},\alpha_{j,\I})-(\alpha_{j,\J},\alpha_{j,\J})} \big[[E_{j,\J}K_{j,\J}^{-1},B_{i,\J}]_{q_{i,j,\J}^{-1}},B_{\tau i,\J}\big]_{q_{i,j,\J}^{-1}} 
 \\\notag
 &-o_{i,\J}(-1)^{|i,\J|} [a_{ij,\J}]_i q^{(\alpha_{j,\I},\alpha_{j,\I})-(\alpha_{j,\J},\alpha_{j,\J})-(\alpha_{i,\J},\alpha_{i,\J})/2-(\alpha_{i,\J},\alpha_{j,\J})} E_{j,\J}K_{j,\J}^{-1} k_{i,\J} .
\end{align}

Let us prove \eqref{eq:diag3}. We rewrite the first term in the right-hand side of \eqref{eq:diag3} as below
\begin{align*}
&\quad\big[[F_{j,\J},B_{i,\J}]_{q_{i,j,\J}^{-1}},B_{\tau i,\J}\big]_{q_{i,j,\J}^{-1}}
\\
&=\big[[F_{j,\J},F_{i,\J}]_{q_{i,j,\J}^{-1}},B_{\tau i,\J}\big]_{q_{i,j,\J}^{-1}}
\\
&=\big[[F_{j,\J},F_{i,\J}]_{q_{i,j,\J}^{-1}},F_{\tau i,\J}\big]_{q_{i,j,\J}^{-1}}
-o_{i,\J}q^{-(\alpha_{i,\J},\alpha_{i,\J})/2}\big[[F_{j,\J},F_{i,\J}]_{q_{i,j,\J}^{-1}},E_{i,\J}\big] K_{\tau i,\J}^{-1}
\\
&=T_{\bs_{i},\I}(F_{j,\I}) +o_{i,\J} (-1)^{|i,\J|} q^{-(\alpha_{i,\J},\alpha_{i,\J})/2} \left[F_{j,\J},\frac{K_{i,\J}-K_{i,\J}^{-1}}{q_i-q_i^{-1}}\right]_{q_{i,j,\J}^{-1}} K_{\tau i,\J}^{-1}
\\
&=T_{\bs_{i},\I}(F_{j,\I})+o_{i,\J}(-1)^{|i,\J|} [a_{ij,\J}]_i
q^{-(\alpha_{i,\I},\alpha_{i,\I})/2-(\alpha_{i,\J},\alpha_{j,\J})} F_{j,\J}  k_{i,\J}.
\end{align*}
Thus, \eqref{eq:diag3} is equivalent to \eqref{eq:rktwo1}, and hence \eqref{eq:diag3} follows.

Let us prove \eqref{eq:diag4}. By Proposition~\ref{prop:qsrkone}, we have \eqref{eq:diag0} and the following identity
\begin{align}\label{eq:diag5}
(\up''_{i,\J})^{-1} B_{\tau i,\J}  \up''_{i,\J} =  o_{i,\J} q^{(\alpha_{i,\J},\alpha_{i,\J})/2} T_{\bs_{i},\I}(B_{i,\I})k_{i,\J} .
\end{align}
Via a case-by-case checking, we have
\begin{align}\label{eq:diag6}
(\alpha_{i,\J},\alpha_{i,\J})+2(\alpha_{i,\J}, \alpha_{j,\J})
=\big((\alpha_{j,\I},\alpha_{j,\I})-(\alpha_{j,\J},\alpha_{j,\J})\big)/2.
\end{align}
By \cref{paritylemma}, $o_{i,\I}=o_{i,\J}$. Using \eqref{eq:rktwo2} and \eqref{eq:diag0}, \eqref{eq:diag5}-\eqref{eq:diag6}, we have
\begin{align*}
& (\up''_{i,\J})^{-1} \big[[E_{ j,\J}K_{j,\J}^{-1},B_{i,\J}]_{q_{i,j,\J}^{-1}},B_{\tau i,\J}\big]_{q_{i,j,\J}^{-1}} \up''_{i,\J} 
\\
=& (-1)^{|i,\J|}q^{(\alpha_{i,\J},\alpha_{i,\J})} 
\big[[E_{j,\J}K_{j,\J}^{-1},T_{\bs_{i},\I}(B_{\tau i,\I})k_{\tau i,\J}]_{q_{i,j,\J}^{-1}},
T_{\bs_{i},\I}(B_{ i,\I})k_{i,\J}\big]_{q_{i,j,\J}^{-1}}
\\
=& (-1)^{|i,\J|} \big[[E_{j,\J} K_{j,\J}^{-1},T_{\bs_{i},\I}(B_{\tau i,\I}) ]_{q_{i,j,\J}^{-1}},
T_{\bs_{i},\I}(B_{ i,\I}) \big]_{q_{i,j,\J}^{-1}}
\\
=& (-1)^{|i,\J|}
\big[[E_{ j,\J}K_{j,\J}^{-1},T_{\bs_{i},\I}(F_{\tau i,\I})]_{q_{i,j,\J}^{-1}},
T_{\bs_{i},\I}(F_{ i,\I})\big]_{q_{i,j,\J}^{-1}}
\\
&-(-1)^{|i,\J|} o_{i,\J} q^{-(\alpha_{i,\J},\alpha_{i,\J})/2} 
\Big[\big[E_{ j,\J}K_{j,\J}^{-1},T_{\bs_{i},\I}(F_{\tau i,\I})\big]_{q_{i,j,\J}^{-1}},
T_{\bs_{i},\I}(E_{\tau i,\I}K_{i,\I}^{-1})\Big]_{q_{i,j,\J}^{-1}}
\\
=& (-1)^{|i,\J|}q^{-2(\alpha_{i,\J},\alpha_{i,\J})} 
\big[[E_{ j,\J}K_{j,\J}^{-1},E_{\tau i,\J} K_{\tau i,\J}^{-1}]_{q_{i,j,\J}^{-1}},
E_{ i,\J} K_{i,\J}^{-1} \big]_{q_{i,j,\J}^{-1}}
\\
&-  o_{i,\J} q^{-(\alpha_{i,\J},\alpha_{i,\J})/2} 
\Big[\big[E_{ j,\J},E_{\tau i,\J} \big]_{q_{i,j,\J}^{-1}},
F_{\tau i,\J} \Big]  K_{j,\J}^{-1}  K_{i,\J}
\\
=& q^{-2(\alpha_{i,\J},\alpha_{i,\J})-4(\alpha_{i,\J},\alpha_{j,\J})} 
\big[E_{ i,\J}, [E_{\tau i,\J}, E_{ j,\J}]_{q_{i,j,\J}} \big]_{q_{i,j,\J} } 
K_{j,\J}^{-1}K_{\tau i,\J}^{-1}K_{i,\J}^{-1}
\\
&-o_{i,\J} q^{-(\alpha_{i,\J},\alpha_{i,\J})/2} 
 \bigg[E_{ j,\J},\frac{K_{\tau i,\J}-K_{\tau i,\J}^{-1}}{q_{\tau i,\J}-q_{\tau i,\J}^{-1}} \bigg]_{q_{i,j,\J}^{-1}}
   K_{j,\J}^{-1}  K_{i,\J}
\\
=& q^{-(\alpha_{j,\I},\alpha_{j,\I})+(\alpha_{j,\J},\alpha_{j,\J})} 
\big[E_{ i,\J}, [E_{\tau i,\J}, E_{ j,\J}]_{q_{i,j,\J}} \big]_{q_{i,j,\J} } 
K_{j,\J}^{-1}K_{\tau i,\J}^{-1}K_{i,\J}^{-1}
\\
&+o_{i,\J}(-1)^{|i,\J|}[a_{ij,\J}]_i q^{-(\alpha_{i,\J},\alpha_{i,\J})/2-(\alpha_{i,\J},\alpha_{j,\J})} E_{j,\J}K_{j,\J}^{-1} k_{i,\J}.
\end{align*}
This computation shows that \eqref{eq:diag4} holds. The proof is completed.  
\end{proof}


\begin{proposition}
\label{prop:Tirktwo3}
Let $\I,\J\in \DiAIII$ such that $\J=\bs_{i,\I}(\I)$. For any $j\in \I_\circ$ such that $\tau i\sim j\sim i$, we have 
\begin{align*} 
&\up_{i,\J} \, T'_{\bs_{i},\I}(B_{j,\I})\, \up_{i,\J}^{-1}
\\
=& \big[B_{\tau i,\J},[B_{i,\J},B_{j,\J}]_{q_{i,j,\J}^{-1}}\big]_{q_{i,j,\J}^{-1}} 
-o_{i,\J} [a_{ij,\J}]_i q^{-(\alpha_{i,\J},\alpha_{i,\J})/2-(\alpha_{i,\J},\alpha_{j,\J})}B_{j,\J} k_{\tau i,\J}.
\end{align*}
\end{proposition}

\begin{proof}
Recall that the case $\tau i\sim j\sim i$ occurs only if $\I_\bu=\emptyset,\tau j=j$. It suffices to prove the following two relations
\begin{align}\label{eq:diag7}
&\up_{i,\J} \, T'_{\bs_{i},\I}(F_{j,\I})\, \up_{i,\J}^{-1}
\\\notag
=& \big[B_{\tau i,\J},[B_{i,\J},F_{j,\J}]_{q_{i,j,\J}^{-1}}\big]_{q_{i,j,\J}^{-1}} 
-o_{i,\J} [a_{ij,\J}]_i q^{-(\alpha_{i,\J},\alpha_{i,\J})/2-(\alpha_{i,\J},\alpha_{j,\J})}F_{j,\J} k_{\tau i,\J},
\\\label{eq:diag8}
&q^{(\alpha_{j,\J},\alpha_{j,\J})-(\alpha_{j,\I},\alpha_{j,\I})} \up_{i,\J} \, T'_{\bs_{i},\I}(E_{j,\I} K_{j,\I}^{-1})\, \up_{i,\J}^{-1}
\\\notag
=&\big[B_{\tau i,\J},[B_{i,\J},E_{j,\J}K_{j,\J}^{-1}]_{q_{i,j,\J}^{-1}}\big]_{q_{i,j,\J}^{-1}} 
\\\notag
&-o_{i,\J} [a_{ij,\J}]_i q^{-(\alpha_{i,\J},\alpha_{i,\J})/2-(\alpha_{i,\J},\alpha_{j,\J})}E_{j,\J}K_{j,\J}^{-1} k_{\tau i,\J}.
\end{align}

We show \eqref{eq:diag7}. By \cref{Kmatrix} and \cref{lem:rktwo}, we have
\begin{align*}
&\quad \up_{i,\J}^{-1} \big[B_{\tau i,\J},[B_{i,\J},F_{j,\J}]_{q_{i,j,\J}^{-1}}\big]_{q_{i,j,\J}^{-1}} \up_{i,\J}
\\
&= \big[\sigma \tau (B_{i,\J}),[\sigma \tau (B_{\tau i,\J}),F_{j,\J}]_{q_{i,j,\J}^{-1}}\big]_{q_{i,j,\J}^{-1}}
\\
&=\big[F_{\tau i,\J},[F_{i,\J},F_{j,\J}]_{q_{i,j,\J}^{-1}}\big]_{q_{i,j,\J}^{-1}}
 -o_{i,\J} q^{-(\alpha_{i,\J},\alpha_{i,\J})/2}\big[E_{i,\J} K_{\tau i,\J} ,[F_{i,\J},F_{j,\J}]_{q_{i,j,\J}^{-1}}\big]_{q_{i,j,\J}^{-1}}
 \\
&= T'_{\bs_{i},\I}(F_{j,\I})-o_{i,\J} 
q^{-(\alpha_{i,\J},\alpha_{i,\J})/2-(\alpha_{i,\J},\alpha_{j,\J})}
\big[E_{i,\J}  ,[F_{i,\J},F_{j,\J}]_{q_{i,j,\J}^{-1}}\big]K_{\tau i,\J}
\\
&=T'_{\bs_{i},\I}(F_{j,\I})-o_{i,\J} 
q^{-(\alpha_{i,\J},\alpha_{i,\J})/2-(\alpha_{i,\J},\alpha_{j,\J})}
 \left[\frac{K_{i,\J}-K_{i,\J}^{-1}}{q_{i,\J}-q_{i,\J}^{-1}},F_{j,\J}\right]_{q_{i,j,\J}^{-1}} K_{\tau i,\J}
 \\
&=T'_{\bs_{i},\I}(F_{j,\I})+o_{i,\J} 
q^{-(\alpha_{i,\J},\alpha_{i,\J})/2-(\alpha_{i,\J},\alpha_{j,\J})} [a_{ij,\J}]_{i} F_{j,\J} k_{\tau i,\J}.
\end{align*}
This proves \eqref{eq:diag7}.

We next show \eqref{eq:diag8}. By \cref{lem:rktwo}, we have $T'_{\bs_{i},\I}(E_{j,\I} K_{j,\I}^{-1})$ commutes with $\up_{i,\J}$ and hence it suffices to show that
\begin{align}\label{eq:diag9}
q^{(\alpha_{j,\J},\alpha_{j,\J})-(\alpha_{j,\I},\alpha_{j,\I})} T'_{\bs_{i},\I}(E_{j,\I} K_{j,\I}^{-1})
=&\big[B_{\tau i,\J},[B_{i,\J},E_{j,\J}K_{j,\J}^{-1}]_{q_{i,j,\J}^{-1}}\big]_{q_{i,j,\J}^{-1}} 
\\\notag
&-o_{i,\J} [a_{ij,\J}]_i q^{-(\alpha_{i,\J},\alpha_{i,\J})/2-(\alpha_{i,\J},\alpha_{j,\J})}E_{j,\J}K_{j,\J}^{-1} k_{\tau i,\J}.
\end{align}
By a direct computation, we have
\begin{align*}
&\quad \big[B_{\tau i,\J},[B_{i,\J},E_{j,\J}K_{j,\J}^{-1}]_{q_{i,j,\J}^{-1}}\big]_{q_{i,j,\J}^{-1}}
\\
&=(-1)^{|i,\J|} q^{-(\alpha_{i,\J},\alpha_{i,\J})}
\big[E_{ i,\J} K_{\tau i,\J}^{-1},[E_{\tau i,\J} K_{i,\J}^{-1},
E_{j,\J}K_{j,\J}^{-1}]_{q_{i,j,\J}^{-1}}\big]_{q_{i,j,\J}^{-1}}
\\
&\quad -o_{i,\J} q^{-(\alpha_{i,\J},\alpha_{i,\J})/2}
\big[F_{\tau i},[E_{\tau i,\J} K_{i,\J}^{-1},
E_{j,\J}K_{j,\J}^{-1}]_{q_{i,j,\J}^{-1}}\big]_{q_{i,j,\J}^{-1}}
\\
&= q^{-2(\alpha_{i,\J},\alpha_{i,\J})-4(\alpha_{i,\J},\alpha_{j,\J})}
\big[[ E_{j,\J},E_{\tau i,\J} ]_{q_{i,j,\J}},E_{ i,\J} \big]_{q_{i,j,\J}}
K_{\tau i,\J}^{-1}K_{i,\J}^{-1}K_{j,\J}^{-1}
\\
&\quad -o_{i,\J} q^{-(\alpha_{i,\J},\alpha_{i,\J})/2-(\alpha_{i,\J},\alpha_{j,\J})}
\big[F_{\tau i,\J},[E_{\tau i,\J} ,
E_{j,\J}]_{q_{i,j,\J}^{-1}}\big]K_{i,\J}^{-1}K_{j,\J}^{-1}
\\
&= q^{(\alpha_{j,\J},\alpha_{j,\J})-(\alpha_{j,\I},\alpha_{j,\I})}
T'_{\bs_{i},\I}(E_{j,\I} K_{j,\I}^{-1})
+o_{i,\J} [a_{ij,\J}]_i q^{-(\alpha_{i,\J},\alpha_{i,\J})/2-(\alpha_{i,\J},\alpha_{j,\J})}E_{j,\J}K_{j,\J}^{-1} k_{\tau i,\J}.
\end{align*}
This proves \eqref{eq:diag9}. The proof is completed.
\end{proof}

\begin{proposition}
\label{prop:Tirktwo4}
Let $\I,\J\in \DiAIII_\pt$ such that $\J=\bs_{i}(\I)$. 
For any $j\in \I_\circ$ such that $j\sim i$ and $j\not\sim \tau i$, we have 
\[
[B_{i,\J},B_{j,\J}]_{q_{i,j,\J}^{-1}} \up_{i,\J} = \up_{i,\J} T'_{\bs_{i},\I}(B_{j,\I}).
\]
\end{proposition}

\begin{proof} 
The strategy in the proof of this proposition is similar to the one of \cref{prop:Tirktwo3}; hence, we omit the details.
\end{proof}

\subsection{Symmetries $\TT'_0,\TT_0$}
In this subsection, we consider the case $a=0$. In this case, $\I_\bullet=\emptyset$, $\bs_{0,\I}=s_{0,\I}$ and the parity for the node $0$ is even for any $\I$. Hence, $s_{0}(\I)=\I$. We will construct the symmetry $\TT_0$ associated to the node $0$.

 \begin{proposition}
 \label{prop:TT0}
Let $\I\in \DiAIII_\pt$ and $j\in \I$ such that $j\sim i$. We have
\begin{align}
[B_{i,\I},B_{j,\I}]_{q_{i,j,\I}^{-1}} \up_{0,\I} & = \up_{0,\I} T'_{0,\I}(B_{j,\I}),
\\
[B_{j,\I},B_{i,\I}]_{q_{i,j,\I}^{-1}} \up''_{0,\I} & = \up''_{0,\I} T_{0,\I}(B_{j,\I}).
\end{align}

\end{proposition}

\begin{proof}
Note that node $0$ is even, and hence the proof for this proposition is essentially the same as its non-super counterpart in \cite[Theorem 4.7 and 6.1]{WZ23}.
\end{proof}
 
 \subsection{Symmetries $\TT'_\pt,\TT_\pt,\pt\neq 0$}
 In this subsection, we will construct symmetry $\TT_{\pt,\I}$ associated to the node $a$ for $a\neq 0$; see \eqref{AIIIdiagram}. In this case, we have
 \[
 \bs_{\pt,\I}=s_{a} s_{\pt-1}\cdots s_{-\pt+1} s_{-\pt} s_{-\pt+1}\cdots s_{\pt-1} s_{\pt,\I}.
 \]
Set $o_{\pt,\I}=\sqrt{(-1)^{|\pt,\I|}}$. By \eqref{balance}, $B_{\pt,\I}=F_{\pt,\I}-o_{\pt,\I} q^{-(\alpha_{\pt,\I},\alpha_{\pt,\I}+w_\bu \alpha_{-\pt,\I})/2}T_{w_\bu}(E_{-\pt,\I}) K^{-1}_{\pt,\I}$. 

Recall from \eqref{qijJ} that $q_{\pt,\pt+1,\J}=q^{(\alpha_{\pt,\J},\alpha_{\pt+1,\J})}$ for $\J\in \mathcal{D}_{m,n}$.

\begin{lemma}\label{lem:qs1}
Let $\I,\J\in \DiAIII_\pt$ such that $\J=\bs_\pt(\I)$. We have
\begin{align}
\label{eq:qs1}
\begin{split}
T_{\bs_\pt,\I}(F_{\pt+1,\I})
= &
\big[[F_{\pt+1,\J},B_{\pt,\J}]_{q_{\pt,\pt+1,\J}^{-1}}, 
T_{w_\bu}^{-1}(B_{-\pt,\J})\big]_{q_{\pt,\pt+1,s_\pt(\I)}^{-1}}
\\
&-o_{\pt,\J} (-1)^{|\pt,\J|} q^{-(\alpha_{\pt,\J},\alpha_{\pt,\J}+w_\bu \alpha_{-\pt,\J})/2} [c_{\pt,\pt+1,\J}]_\pt\cK_{\pt,\J} F_{\pt+1,\J}.
\end{split}
\end{align}
\end{lemma}

\begin{proof}
On one side, we have
\begin{align*}
T_{\bs_\pt,\I}(F_{\pt+1,\I})&=T_{s_\pt s_{\pt-1} \cdots s_{-\pt} \cdots s_{\pt-1},s_\pt(\I)}\Big([F_{\pt+1,s_\pt(\I)},F_{a,s_\pt(\I)}]_{q_{\pt,\pt+1,s_\pt(\I)}^{-1}}\Big)
\\
&=\big[[F_{\pt+1,\J},F_{\pt,\J}]_{q_{\pt,\pt+1,\J}^{-1}}, 
T_{s_\pt s_{\pt-1} \cdots s_{-\pt} \cdots s_{\pt-1},s_\pt(\I)}
\big(F_{a,s_\pt(\I)}\big)\big]_{q_{\pt,\pt+1,s_\pt(\I)}^{-1}}
\\
&= \big[[F_{\pt+1,\J},F_{\pt,\J}]_{q_{\pt,\pt+1,\J}^{-1}}, 
T_{w_\bu}^{-1}(F_{-\pt,\J})\big]_{q_{\pt,\pt+1,s_\pt(\I)}^{-1}}.
\end{align*}
 (cf. \cite[Lemma A.17]{WZ23} for a similar calculation in the non-super setting.)
 
On the other side, we rewrite the first term on right-hand side of \eqref{eq:qs1} as follows
\begin{align*}
&\quad \big[[F_{\pt+1,\J},B_{\pt,\J}]_{q_{\pt,\pt+1,\J}^{-1}}, 
T_{w_\bu}^{-1}(B_{-\pt,\J})\big]_{q_{\pt,\pt+1,s_\pt(\I)}^{-1}}
\\
&= \big[[F_{\pt+1,\J},F_{\pt,\J}]_{q_{\pt,\pt+1,\J}^{-1}}, 
T_{w_\bu}^{-1}(F_{-\pt,\J})\big]_{q_{\pt,\pt+1,s_\pt(\I)}^{-1}}
\\
&\quad -o_{\pt,\J}q^{-(\alpha_{\pt,\J},\alpha_{\pt,\J}+w_\bu \alpha_{-\pt,\J})/2}
\big[[F_{\pt+1,\J},F_{\pt,\J}]_{q_{\pt,\pt+1,\J}^{-1}}, 
E_{\pt,\J} T_{w_\bu}^{-1}(K_{-\pt,\J}^{-1})\big]_{q_{\pt,\pt+1,s_\pt(\I)}^{-1}}
\\
&= \big[[F_{\pt+1,\J},F_{\pt,\J}]_{q_{\pt,\pt+1,\J}^{-1}}, 
T_{w_\bu}^{-1}(F_{-\pt,\J})\big]_{q_{\pt,\pt+1,s_\pt(\I)}^{-1}}
\\
&\quad +o_{\pt,\J}(-1)^{|\pt,\J|} q^{-(\alpha_{\pt,\J},\alpha_{\pt,\J}+w_\bu \alpha_{-\pt,\J})/2}
\left[F_{\pt+1,\J},\frac{K_{\pt,\J}-K_{\pt,\J}^{-1}}{q_{\pt,\J}-q_{\pt,\J}^{-1}}\right]_{q_{\pt,\pt+1,\J}^{-1}} T_{w_\bu}^{-1}(K_{-\pt,\J}^{-1})
\\
&=T_{\bs_\pt,\I}(F_{\pt+1,\I})
+o_{\pt,\J} (-1)^{|\pt,\J|} q^{-(\alpha_{\pt,\J},\alpha_{\pt,\J}+w_\bu \alpha_{-\pt,\J})/2} [c_{\pt,\pt+1,\J}]_\pt
\cK_{\pt,\J} F_{\pt+1,\J}.
\end{align*}
Hence, \eqref{eq:qs1} follows.
\end{proof}

\begin{proposition}\label{prop:qs}
Let $\I,\J\in \DiAIII_\pt$ such that $\J=\bs_\pt(\I)$. We have
\begin{align}
\label{eq:qs3}
\begin{split}
\up''_{\pt,\J} T_{\bs_\pt,\I}(B_{\pt+1,\I}) \big(\up''_{\pt,\J}\big)^{-1}
= &
\big[[B_{\pt+1,\J},B_{\pt,\J}]_{q_{\pt,\pt+1,\J}^{-1}}, 
T_{w_\bu}^{-1}(B_{-\pt,\J})\big]_{q_{\pt,\pt+1,s_\pt(\I)}^{-1}}
\\
&-o_{\pt,\J} (-1)^{|\pt,\J|} q^{-(\alpha_{\pt,\J},\alpha_{\pt,\J}+w_\bu \alpha_{-\pt,\J})/2} [c_{\pt,\pt+1,\J}]_\pt\cK_{\pt,\J} B_{\pt+1,\J}.
\end{split}
\end{align}
\end{proposition}

\begin{proof}
By \cref{lem:rktwo}, $\up''_{\pt,\J}$ commutes with $T_{\bs_\pt,\I}(F_{\pt+1,\I})$. Hence, we have
\begin{align*}
&\quad \up''_{\pt,\J} T_{\bs_\pt,\I}(B_{\pt+1,\I}) \big(\up''_{\pt,\J}\big)^{-1}
\\
&=T_{\bs_\pt,\I}(F_{\pt+1,\I})-o_{\pt+1,\I} q^{-(\alpha_{\pt+1,\I},\alpha_{\pt+1,\I})/2}
\up''_{\pt,\J} T_{\bs_\pt,\I}\big(E_{-\pt-1,\I} K_{\pt+1,\I}^{-1}\big) \big(\up''_{\pt,\J}\big)^{-1}.
\end{align*}
Apply \cref{lem:qs1} to the above identity, and then it remains to show that 
\begin{align}
\label{eq:qs4}
\begin{split}
&\quad \up''_{\pt,\J} T_{\bs_\pt,\I}\big(E_{-\pt-1,\I} K_{\pt+1,\I}^{-1}\big) \big(\up''_{\pt,\J}\big)^{-1}
\\
&=\big[[E_{-\pt-1,\J} K_{\pt+1,\J}^{-1},B_{\pt,\J}]_{q_{\pt,\pt+1,\J}^{-1}}, 
T_{w_\bu}^{-1}(B_{-\pt,\J})\big]_{q_{\pt,\pt+1,s_\pt(\I)}^{-1}}
\\
&\quad -o_{\pt,\J} (-1)^{|\pt,\J|} q^{-(\alpha_{\pt,\J},\alpha_{\pt,\J}+w_\bu \alpha_{-\pt,\J})/2} [c_{\pt,\pt+1,\J}]_\pt\cK_{\pt,\J} E_{-\pt-1,\J} K_{\pt+1,\J}^{-1}.
\end{split}
\end{align}

We prove \eqref{eq:qs4}. Note that $\up_{\pt,\I}$ is fixed by $T_{w_\bu}$ and then $\up''_{\pt,\J}$ is also fixed by $T_{w_\bu}$. Now \cref{prop:qsrkone} implies that 
\begin{align}
\label{eq:qs5}
\big(\up''_{\pt,\J}\big)^{-1}  B_{\pt,\J} \up''_{\pt,\J}=o_{\pt,\J} q^{(\alpha_{\pt,\J},\alpha_{\pt,\J}-w_\bu \alpha_{-\pt,\J})/2} T_{w_\bu}T_{\bs_\pt w_\bu,\I}(B_{\pt,\I}) \cK_{\pt,\J}^{-1}.
\end{align}
Via a case-by-case checking, we have
\begin{align}\label{eq:qs6}
\begin{split}
&(\alpha_{\pt,\J},w_\bu \alpha_{-\pt,\J})=(\alpha_{a,s_\pt(\I)},\alpha_{\pt+1,s_\pt(\I)}),
\\
&(\alpha_{\pt,\J},\alpha_{\pt,\J}+w_\bu \alpha_{-\pt,\J})+(\alpha_{\pt,\J},\alpha_{\pt+1,\J})=0.
\end{split}
\end{align}
Using \cref{lem:rktwo} and \eqref{eq:qs5}-\eqref{eq:qs6}, we have
\begin{align*}
&\quad \big(\up''_{\pt,\J}\big)^{-1} \big[[E_{-\pt-1,\J} K_{\pt+1,\J}^{-1},B_{\pt,\J}]_{q_{\pt,\pt+1,\J}^{-1}}, 
T_{w_\bu}^{-1}(B_{-\pt,\J})\big]_{q_{\pt,\pt+1,s_\pt(\I)}^{-1}} \up''_{\pt,\J}
\\
&=(-1)^{|\pt,\J|} q^{(\alpha_{\pt,\J},\alpha_{\pt,\J}-w_\bu \alpha_{-\pt,\J})} \times
\\
&\qquad \times \big[[E_{-\pt-1,\J}K_{\pt+1,\J}^{-1},T_{w_\bu}T_{\bs_\pt w_\bu,\I}(B_{\pt,\I}) \cK_{\pt,\J}^{-1}]_{q_{\pt,\pt+1,\J}^{-1}}, 
T_{\bs_\pt w_\bu,\I}(B_{-\pt,\I}) \cK_{\pt,\J}\big]_{q_{\pt,\pt+1,s_\pt(\I)}^{-1}} 
\\
&=(-1)^{|\pt,\J|} q^{(\alpha_{\pt,\J},\alpha_{\pt,\J}-w_\bu \alpha_{-\pt,\J})} \times
\\
&\qquad \times \big[[E_{-\pt-1,\J} K_{\pt+1,\J}^{-1},T_{w_\bu}T_{\bs_\pt w_\bu,\I}(F_{\pt,\I}) \cK_{\pt,\J}^{-1}]_{q_{\pt,\pt+1,\J}^{-1}}, 
T_{\bs_\pt w_\bu,\I}(B_{-\pt,\I}) \cK_{\pt,\J}\big]_{q_{\pt,\pt+1,s_\pt(\I)}^{-1}} 
\\
&=(-1)^{|\pt,\J|} q^{-(\alpha_{\pt,\J},\alpha_{\pt,\J}+w_\bu \alpha_{-\pt,\J})} \times
\\
&\qquad \times \big[[E_{-\pt-1,\J} , T_{w_\bu}(E_{-\pt,\J}) ]_{q_{\pt,\pt+1,\J}^{-1}}K_{\pt+1,\J}^{-1}K_{\pt,\J}^{-1}, 
E_{\pt,\J} T_{w_\bu}(K_{-\pt,\J}^{-1})\big]_{q_{\pt,\pt+1,s_\pt(\I)}^{-1}} 
\\
&\quad -o_{\pt,\I}(-1)^{|\pt,\J|} q^{-(\alpha_{\pt,\J},w_\bu \alpha_{-\pt,\J})-(\alpha_{\pt,\J}, \alpha_{\pt,\J}+w_\bu \alpha_{-\pt,\J})/2} \times
\\
&\qquad \times \big[[E_{-\pt-1,\J} , T_{w_\bu}(E_{-\pt,\J}) ]_{q_{\pt,\pt+1,\J}^{-1}}K_{\pt+1,\J}^{-1}K_{\pt,\J}^{-1}, 
T_{w_\bu}(F_{-\pt,\J}) K_{\pt,\J}^{2} \big]_{q_{\pt,\pt+1,s_\pt(\I)}^{-1}} 
\\
&= q^{-(\alpha_{\pt,\J},2\alpha_{\pt,\J}+w_\bu \alpha_{-\pt,\J})-2(\alpha_{\pt,\J},\alpha_{\pt+1,\J})-(\alpha_{a,s_\pt(\I)},\alpha_{a+1,s_{a}(\I)}) } \times
\\
&\qquad \times \big[E_{\pt,\J}, [T_{w_\bu}(E_{-\pt,\J} ), E_{-\pt-1,\J} ]_{q_{\pt,\pt+1,\J}} 
\big]_{q_{\pt,\pt+1,s_\pt(\I)}} 
K_{\pt+1,\J}^{-1}K_{\pt,\J}^{-1}T_{w_\bu}(K_{-\pt,\J}^{-1})
\\
&\quad -o_{\pt,\J}(-1)^{|\pt,\J|} q^{-(\alpha_{\pt,\J}, \alpha_{\pt,\J}+w_\bu \alpha_{-\pt,\J})/2} \times
\\
&\qquad \times \left[E_{-\pt-1,\J} , T_{w_\bu}\left(\frac{K_{-\pt,\J}-K_{-\pt,\J}^{-1}}{q_{-\pt,\J}-q_{-\pt,\J}^{-1}}\right) \right]_{q_{\pt,\pt+1,\J}^{-1}} K_{\pt+1,\J}^{-1}K_{\pt,\J} 
\\
&=T_{\bs_\pt,\I}\big(E_{-\pt-1,\I} K_{\pt+1,\I}^{-1}\big)
 +o_{\pt,\J} (-1)^{|\pt,\J|} q^{-(\alpha_{\pt,\J},\alpha_{\pt,\J}+w_\bu \alpha_{-\pt,\J})/2} [c_{\pt,\pt+1,\J}]_\pt 
\cK_{\pt,\J} E_{-\pt-1,\J} K_{\pt+1,\J}^{-1}.
\end{align*}
This computation shows that \eqref{eq:qs4} holds. The proof is completed.
\end{proof}

\begin{lemma}\label{lem:qsqs}
Let $\I,\J\in \DiAIII_\pt$ such that $\J=\bs_\pt(\I)$. We have
\begin{align}
\label{eq:qsqs1}
\begin{split}
T'_{\bs_\pt,\I}(E_{-\pt-1,\I}K_{\pt+1,\I}^{-1}) 
= & \big[T_{w_\bu}(B_{-\pt,\J}), [B_{\pt,\J},E_{-\pt-1,\J}K_{\pt,\J}^{-1} ]_{q_{\pt,\pt+1,\J}^{-1}}\big]_{q_{\pt,\pt+1,s_\pt(\I)}^{-1}}
\\
&-o_{\pt,\J}  q^{-(\alpha_{\pt,\J},\alpha_{\pt,\J}+w_\bu \alpha_{-\pt,\J})/2} [c_{\pt,\pt+1,\J}]_\pt 
E_{-\pt-1,\J} K_{\pt+1,\J}^{-1} \cK_{\pt,\J}^{-1}.
\end{split}
\end{align}
\end{lemma}

\begin{proof}
Note that $T'_{\bs_\pt}$ commutes with $T_{w_\bu}$. By applying $T_{w_\bu}^{-2}$, it is clear that \eqref{eq:qsqs1} is equivalent to
\begin{align}
\label{eq:qsqs1'}
\begin{split}
T'_{\bs_\pt,\I}(E_{-\pt-1,\I}K_{\pt+1,\I}^{-1}) 
= & \big[T_{w_\bu}^{-1}(B_{-\pt,\J}), [T_{w_\bu}^{-2}(B_{\pt,\J}),E_{-\pt-1,\J}K_{\pt,\J}^{-1} ]_{q_{\pt,\pt+1,\J}^{-1}}\big]_{q_{\pt,\pt+1,s_\pt(\I)}^{-1}}
\\
&-o_{\pt,\J}  q^{-(\alpha_{\pt,\J},\alpha_{\pt,\J}+w_\bu \alpha_{-\pt,\J})/2} [c_{\pt,\pt+1,\J}]_\pt 
E_{-\pt-1,\J} K_{\pt+1,\J}^{-1} \cK_{\pt,\J}^{-1}.
\end{split}
\end{align}
We show \eqref{eq:qsqs1'}. On one side, we have
\[
T'_{\bs_\pt,\I}(E_{-\pt-1,\I})=\big[[E_{-\pt-1,\J}, T_{w_\bu}^{-1}(E_{-\pt,\J})]_{q_{\pt,\pt+1,\J} },
E_{\pt,\J}\big]_{q_{\pt,\pt+1,s_\pt(\I)} }.
\] 
On the other side, we have
\begin{align*}
&\quad \big[T_{w_\bu}^{-1}(B_{-\pt,\J}), [T_{w_\bu}^{-2}(B_{\pt,\J}),E_{-\pt-1,\J}K_{\pt,\J}^{-1} ]_{q_{\pt,\pt+1,\J}^{-1}}\big]_{q_{\pt,\pt+1,s_\pt(\I)}^{-1}}
\\
&=(-1)^{|\pt,\J|} q^{-(\alpha_{\pt,\J},\alpha_{\pt,\J}+w_\bu \alpha_{-\pt,\J})}\times
\\
&\qquad \times \big[E_{\pt,\J} T_{w_\bu}(K_{-\pt,\J}^{-1}), [T_{w_\bu}^{-1}(E_{-\pt,\J})K_{\pt,\J}^{-1},E_{-\pt-1,\J}K_{\pt+1,\J}^{-1} ]_{q_{\pt,\pt+1,\J}^{-1}}\big]_{q_{\pt,\pt+1,s_\pt(\I)}^{-1}}
\\
&\quad -o_{\pt,\J}  q^{-(\alpha_{\pt,\J},\alpha_{\pt,\J}+w_\bu \alpha_{-\pt,\J})/2}\big[T_{w_\bu}^{-1}(F_{-\pt,\J}), [T_{w_\bu}^{-1}(E_{-\pt,\J}),E_{-\pt-1,\J}]_{q_{\pt,\pt+1,\J}^{-1}}K_{\pt,\J}^{-1}K_{\pt+1,\J}^{-1} \big]_{q_{\pt,\pt+1,s_\pt(\I)}^{-1}}
\\
&=q^{-(\alpha_{\pt,\J},2\alpha_{\pt,\J}+w_\bu \alpha_{-\pt,\J})-2(\alpha_{\pt,\J},\alpha_{\pt+1,\J})-(\alpha_{a,s_\pt(\I)},\alpha_{a+1,s_{a}(\I)}) } \times
\\
&\qquad \times\big[[E_{-\pt-1,\J}, T_{w_\bu}^{-1}(E_{-\pt,\J})]_{q_{\pt,\pt+1,\J} },
E_{\pt,\J}\big]_{q_{\pt,\pt+1,s_\pt(\I)} } K_{\pt+1,\J}^{-1} K_{\pt,\J}^{-1}T_{w_\bu}(K_{-\pt,\J}^{-1})
\\
&\quad +o_{\pt,\J}  q^{-(\alpha_{\pt,\J},\alpha_{\pt,\J}+w_\bu \alpha_{-\pt,\J})/2} \left[T_{w_\bu}^{-1}\bigg(\frac{K_{-\pt,\J}-K_{-\pt,\J}^{-1}}{q_{\pt,\J}-q_{\pt,\J}^{-1}}\bigg),E_{-\pt-1,\J}\right]_{q_{\pt,\pt+1,\J}^{-1}}  K_{\pt,\J}^{-1}K_{\pt+1,\J}^{-1}
\\
&=T'_{\bs_\pt,\I}(E_{-\pt-1,\I}K_{\pt+1,\I}^{-1})
+o_{\pt,\J}  q^{-(\alpha_{\pt,\J},\alpha_{\pt,\J}+w_\bu \alpha_{-\pt,\J})/2} [c_{\pt,\pt+1,\J}]_\pt 
E_{-\pt-1,\J} K_{\pt+1,\J}^{-1}\cK_{\pt,\J}^{-1},
\end{align*}
where we used \eqref{eq:qs6} in the last equality. This proves \eqref{eq:qsqs1'} and then \eqref{eq:qsqs1} follows.
\end{proof}

\begin{proposition}\label{prop:qsqs}
Let $\I,\J\in \DiAIII_\pt$ such that $\J=\bs_\pt(\I)$. We have
\begin{align}
\label{eq:qsqs2}
\begin{split}
\up_{\pt,\J} T'_{\bs_\pt,\I}(B_{\pt+1,\I}) \big(\up_{\pt,\J}\big)^{-1}
= & \big[T_{w_\bu}(B_{-\pt,\J}), [B_{\pt,\J},B_{\pt+1,\J}]_{q_{\pt,\pt+1,\J}^{-1}}\big]_{q_{\pt,\pt+1,s_\pt(\I)}^{-1}}
\\
&-o_{\pt,\J}  q^{-(\alpha_{\pt,\J},\alpha_{\pt,\J}+w_\bu \alpha_{-\pt,\J})/2} [c_{\pt,\pt+1,\J}]_\pt  B_{\pt+1,\J}\cK_{\pt,\J}^{-1}.
\end{split}
\end{align}
\end{proposition}

\begin{proof}
Thanks to \cref{lem:qsqs}, it suffices to show the following relation
\begin{align}
\label{eq:qsqs3}
\begin{split}
\up_{\pt,\J} T'_{\bs_\pt,\I}(F_{\pt+1,\I}) \big(\up_{\pt,\J}\big)^{-1}
= & \big[T_{w_\bu}(B_{-\pt,\J}), [B_{\pt,\J},F_{\pt+1,\J}]_{q_{\pt,\pt+1,\J}^{-1}}\big]_{q_{\pt,\pt+1,s_\pt(\I)}^{-1}}
\\
&-o_{\pt,\J}  q^{-(\alpha_{\pt,\I},\alpha_{\pt,\I}+w_\bu \alpha_{-\pt,\I})/2} [c_{\pt,\pt+1,\J}]_\pt  F_{\pt+1,\J}\cK_{\pt,\J}^{-1}.
\end{split}
\end{align}
By \cref{Kmatrix}, we have
\begin{align*}
& \quad \up_{\pt,\J}^{-1} \big[T_{w_\bu}(B_{-\pt,\J}), [B_{\pt,\J},F_{\pt+1,\J}]_{q_{\pt,\pt+1,\J}^{-1}}\big]_{q_{\pt,\pt+1,s_\pt(\I)}^{-1}} \up_{\pt,\J}
\\
&=\big[T_{w_\bu}\big(\sigma \tau (B_{\pt,\J})\big), [\sigma \tau(B_{-\pt,\J}),F_{\pt+1,\J}]_{q_{\pt,\pt+1,\J}^{-1}}\big]_{q_{\pt,\pt+1,s_\pt(\I)}^{-1}}
\\
&=\big[T_{w_\bu}(F_{-\pt,\J}), [F_{\pt,\J},F_{\pt+1,\J}]_{q_{\pt,\pt+1,\J}^{-1}}\big]_{q_{\pt,\pt+1,s_\pt(\I)}^{-1}}
\\
&\quad -o_{\pt,\J} q^{-(\alpha_{\pt,\J},\alpha_{\pt,\J}+w_\bu \alpha_{-\pt,\J})/2} \big[T_{w_\bu}(K_{-\pt,\J}) E_{a ,\J}, [F_{\pt,\J},F_{\pt+1,\J}]_{q_{\pt,\pt+1,\J}^{-1}}\big]_{q_{\pt,\pt+1,s_\pt(\I)}^{-1}}
\\
&=T'_{\bs_\pt,\I}(F_{\pt+1,\I})-o_{\pt,\J} q^{-(\alpha_{\pt,\J},\alpha_{\pt,\J}+w_\bu \alpha_{-\pt,\J})/2} \left[\frac{K_{a ,\J}-K_{\pt,\J}^{-1}}{q_{\pt,\J}-q_{\pt,\J}^{-1}},F_{\pt+1,\J}\right]_{q_{\pt,\pt+1,\J}^{-1}} T_{w_\bu}(K_{-\pt,\J}) 
\\
&=T'_{\bs_\pt,\I}(F_{\pt+1,\I})+o_{\pt,\J} q^{-(\alpha_{\pt,\J},\alpha_{\pt,\J}+w_\bu \alpha_{-\pt,\J})/2} [c_{\pt,\pt+1,\J}]_\pt F_{\pt+1,\J} \cK_{\pt,\J}^{-1}.
\end{align*}
This proves \eqref{eq:qsqs3} and hence \eqref{eq:qsqs2} follows.
\end{proof}

\section{Relative braid group actions}
\label{sec:factorization}
 In this section, we first show that $\TT_{\wseq}$ sends $B_{i,\I}$ to $B_{i,\J}$ if $w(\alpha_{i,\I})=\alpha_{j,\J}$. We then extend the factorization result of quasi $K$-matrices to quantum supersymmetric pairs of type sAIII, generalizing \cite{DK19, WZ23}. Finally, we use the factorization property of quasi $K$-matrices to show that $\TT_i,\TT'_{i}$ satisfy braid relations in the underlying relative Coxeter groupoid.

\subsection{A basic property}
Given $w\in \reW$ with a reduced expression
\(
\wseq=\bs_{i_1}\bs_{i_2} \cdots \bs_{i_k,\I},
\)
we define 
\[
\TT_{\wseq}:=\TT_{i_1,\I_1}\circ \TT_{i_2,\I_{2}} \circ \cdots  \circ \TT_{i_k,\I_k}.
\]
where $\I_k=\I$ and $\I_j=\bs_{j}\bs_{j+1}\cdots \bs_{k-1}(\I)$ for any $1\leq j \leq k-1$. For simplicity, we may also write $\TT_{\wseq}:=\TT_{i_1}\circ \TT_{i_2} \circ \cdots  \circ \TT_{i_k,\I}$ instead.
\begin{theorem}
\label{thm:property}
    Suppose $\J=w(\I)$ and $i\in \I_\circ$ such that $w(\alpha_{i,\I})=\alpha_{j,\J}$ for some $j\in \J_\circ$. Then we have 
    \[
    \TT_{\wseq}(B_{i,\I})=B_{j,\J}
    \]
 for some reduced expression of $w$.
\end{theorem}

The remainder of this subsection is devoted to the proof of \cref{thm:property}, following a strategy similar to its non-super analog \cite[Theorem 7.13]{WZ23}.  
To this end, we first treat the real rank \(2\) case (see \cref{def:rank}), namely we consider a Satake sub-diagram $(\I_\bu\cup \{i,j,\tau i,\tau j\},\tau)$ for some $i,j \in \I_\circ, i,j\geqslant 0$. For any $\I\in \DiAIII_\pt$, there are three possibilities of real rank \(2\) sub-diagrams
\begin{itemize}
\item  $i\nsim j$;
\item  $i\sim j$ and $\{i,j\}\neq \{\pt,\pt+1\}$;
    \item $i\sim j$ and $\{i,j\}=\{\pt,\pt+1\}$.
\end{itemize}

The next two lemmas are super analogs of \cite[Lemmas 7.5 and 7.1]{WZ23}, respectively.

\begin{lemma} 
\label{tata1}
    Suppose $\I\in \DiAIII_\pt$ and $\J=\bs_{\pt+1}(\I)$, then  
    \begin{gather*}
    \bs_{\pt}\bs_{\pt+1}\bs_{\pt}(\alpha_{\pm (\pt+1),\I})=\alpha_{\mp (\pt+1),\J},\quad 
     \bs_{\pt+1}\bs_{\pt}\bs_{\pt+1}(\alpha_{\pm \pt,\I})=\alpha_{\pm \pt,\I}.
    \end{gather*}
    Thus we have
    \begin{gather}
        \label{lamb1}
    T_{\bs_{\pt}}T_{\bs_{\pt+1}}T_{\bs_{\pt}}(B_{\pm(\pt+1),\I})=B_{\mp (\pt+1),\J},\quad T_{\bs_{\pt+1}}T_{\bs_{\pt}}T_{\bs_{\pt+1}}(B_{\pm \pt,\I})=B_{\pm \pt,\I}.
    \end{gather}
\end{lemma}
\begin{proof}
    We note that $\bs_{\pt}\bs_{\pt+1}\bs_{\pt}(\I)=\J$ and $\bs_{\pt+1}\bs_{\pt}\bs_{\pt+1}(\I)=\I$ by \cref{bsiI}. We only prove the first equality in \cref{lamb1} and the remaining one can follows similarly. By a direct computation on weights we obtain that $\bs_{\pt}\bs_{\pt+1}\bs_{\pt}(\alpha_{\pm (\pt+1),\I})=\alpha_{\mp (\pt+1),\J}$. Thus by \cite[Lem. 3.14]{Sh25} (see also \cite[39.2]{Lubook}) we have
    \[
    T_{\bs_{\pt}}T_{\bs_{\pt+1}}T_{\bs_{\pt}}(X_{\pm (\pt+1),\I})=X_{\mp(\pt+1),\J}\qquad \text{ for }X=E,F,K.
    \]
    Since $\bs_{\pt}\bs_{\pt+1}\bs_{\pt} w_\bu(\I)=w_\bu\bs_{\pt}\bs_{\pt+1}\bs_{\pt}(\I)$, we conclude that
    \begin{align*}
         T_{\bs_{\pt}}T_{\bs_{\pt+1}}T_{\bs_{\pt}}(B_{\pt+1,\I})&\overset{\cref{eq:Uibvs}}{=}T_{\bs_{\pt}}T_{\bs_{\pt+1}}T_{\bs_{\pt}}(F_{\pt+1,\I} + \va^\dm_{\pt+1,\I} T_{w_\bu}(E_{-\pt-1,\I}) K^{-1}_{\pt+1,\I})
         \\
         &=F_{-\pt-1,\J}+\va^\dm_{\pt+1,\I} T_{w_\bu}(E_{\pt+1,\J})K^{-1}_{-\pt-1,\J}.
    \end{align*}
    Since $\J=\bs_{\pt+1}(\I)$, we have $\va^\dm_{\pt+1,\I}=\va^\dm_{-\pt-1,\J}$ by \cref{balance} and \cref{lem:iparity}. This finishes the proof.
\end{proof}

\begin{lemma}
\label{tata2}
   Suppose $\I\in \DiAIII_\pt$ with $0< i,j\in \I_\circ $ such that  $i\sim j$ and $\{i,j\}\neq \{\pt,\pt+1\}$. Then we have
   \[
   \bs_i\bs_j(\alpha_{\pm i,\I})=\alpha_{\pm j,\J},\qquad T_{\bs_i}T_{\bs_j}(B_{\pm i,\I})=B_{\pm j, \J}.
   \]
   where $\J=\bs_i\bs_j(\I)$
\end{lemma}
\begin{proof}
    Similar to the proof of \cref{tata1}.
\end{proof}

The following proposition is a super analog of \cite[Propositions 7.2 and 7.6]{WZ23}.

\begin{proposition}
\label{prop:relamb1}
    Retain the set-up in \cref{tata1} and \cref{tata2}. Then  
     \begin{gather}
        \label{relamb1}
    \TT_{\bs_{\pt}}\TT_{\bs_{\pt+1}}\TT_{\bs_{\pt}}(B_{\pm(\pt+1),\I})=B_{\mp (\pt+1),\bs_{\pt}(\I)},\quad \TT_{\bs_{\pt+1}}\TT_{\bs_{\pt}}\TT_{\bs_{\pt+1}}(B_{\pm \pt,\I})=B_{\pm \pt,\I}.\\
    \TT_{\bs_i}\TT_{\bs_j}(B_{\pm i,\I})=B_{\pm j,\J}. \label{relamb2}
    \end{gather}
\end{proposition}
\begin{proof}
We only prove the second equality in \cref{relamb1} and the rest follow similarly.
    By \cref{thm:ibraid}, the second equality in \cref{relamb1} is equivalent to
    \[
    \TT'_{\bs_{\pt+1}}\TT'_{\bs_{\pt}}\TT'_{\bs_{\pt+1}}(B_{\pm \pt,\I})=B_{\pm \pt,\I}.
    \]
    By \cref{Iwasawabasis}, we can write
    \[
    \TT'_{\bs_{\pt+1}}\TT'_{\bs_{\pt}}\TT'_{\bs_{\pt+1}}(B_{\pm \pt,\I})-B_{\pm \pt,\I}=\sum_{J\in \mathcal{J}} A_J B_J
    \]
    for some $A_J\in \Ub^+(\I) \U^{\imath 0}(\I)$. Then by a weight argument entirely similar to the proof of \cite[Prop. 7.2]{WZ23} one concludes that $\sum_{J\in \mathcal{J}} A_J B_J=0$. This finishes the proof.
\end{proof}

 Now we are ready to prove \cref{thm:property} in general.

 \begin{proof}[Proof of \cref{thm:property}]
The strategy of the proof is adapted from the well-known quantum group analogue; cf. \cite[Lemma~8.20]{Ja96}.  
We reduce the argument to the rank-two cases established earlier, and then complete the proof by induction on $\ell_\circ(w)$.  
By \cref{prop:relamb1}, the statement is already known for arbitrary rank-two Satake subdiagrams
\(
(\I_\bu \cup \{i,\tau i,j,\tau j\},\tau).
\)
(The case \(i\nsim j\) is immediate.)

We now argue by induction on $\ell_\circ(w)$ for \(w\in \reW\), where $\ell_\circ$ denotes the length function on the relative Weyl group $\reW$.  
Following \cite[\S 2.3]{DK19}, for \(i\in \I_\circ\) and \(\I\in \DiAIII_\pt\), we define the relative simple roots by
\begin{align}
 \label{def:talpha}
\talpha_{i,\I}:=\frac{\alpha_{i,\I} +w_\bu(\alpha_{\tau i,\I})}{2}.
\end{align}
By construction, we have \(\talpha_{i,\I}=\talpha_{\tau i,\I}\).  
Moreover, since \(w(\alpha_{i,\I})=\alpha_{j,\J}\) for some \(j\in \J_\circ\), a direct computation shows that
\[
w(\talpha_{i,\I})=\talpha_{j,\J}.
\]
As usual, we write \(\beta>0\) (resp. \(\beta<0\)) for a positive (resp. negative) root in the relative root system.

Suppose now that \(\ell_\circ(w)>0\).  
Then there exists \(k\in \I_\circ\) such that \(w(\talpha_{k,\I})<0\); clearly \(k\neq i\), since \(w(\talpha_{i,\I})>0\).  
Consider the decomposition of \(w\) with respect to the rank-two parabolic subgroupoid \(\langle \bs_i,\bs_k\rangle\).  
Then we may write
\(
w=w'w''
\)
in \(\reW\), where \(w''\in \langle \bs_i,\bs_k\rangle\), the element \(w'\) is the minimal-length representative of the corresponding coset, and
\[
w'(\talpha_{i,\I})>0,\qquad w'(\talpha_{k,\I})>0,
\qquad
\ell_\circ(w)=\ell_\circ(w')+\ell_\circ(w'').
\]
Since \(w(\talpha_{i,\I})>0\) and \(w(\talpha_{k,\I})<0\), it follows that
\[
w''(\talpha_{i,\I})>0,\qquad w''(\talpha_{k,\I})<0.
\]
Hence
\[
w''(\alpha_{i,\I})>0,\qquad w''(\alpha_{k,\I})<0,\qquad
w'(\alpha_{i,\I})>0,\qquad w'(\alpha_{k,\I})>0,
\]
because the positive system of the relative root system is compatible with that of \(\Phi\).

Using the same weight argument as in the proof of \cite[Theorem~7.13]{WZ23}, we deduce that \(w''i\in \J_\circ\) and
\[
w''(\alpha_{i,\I})=\alpha_{w''i,\J}.
\]
By the rank-two case established in \cref{prop:relamb1}, we have
\[
\TT_{\underline{w''}}(B_{i,\I})=B_{w''i,\J}.
\]
for any reduced expression \(\underline{w''}\) of \(w''\).  
Now apply the induction hypothesis to \(w'\).  
There exists a reduced expression \(\underline{w'}\) such that
\(
\underline{w}=\underline{w'}\,\underline{w''}
\)
is a reduced expression for \(w\), and
\begin{align*}
\TT_{\underline{w}}(B_i)
=
\TT_{\underline{w'}}\TT_{\underline{w''}}(B_i)
=
\TT_{\underline{w'}}(B_{w''i})
=
B_{wi}.
\end{align*}
This completes the proof.
 \end{proof}

 \subsection{Factorization of quasi $K$-matrix}
Let $\I\in \DiAIII_\pt$. Given $w\in \reW$ with a reduced expression
\(
\wseq=\bs_{i_1}\bs_{i_2} \cdots \bs_{i_k,\I},
\)
we define (following \cite{DK19,WZ23}), for $1\leq j \leq k$
\begin{gather}
    \label{upj}
    \up^{[j]}:=T_{\bs_{i_1}}T_{\bs_{i_2}}\cdots T_{\bs_{i_{j-1}}}(\up_{i_j,\I_j}), \qquad \text{ where }\I_j=\bs_{j}\bs_{j+1}\dots \bs_{k} (\I),
    \\
    \label{upwseq}
    \up_{\wseq}=\up^{[k]}\up^{[k-1]}\cdots \up^{[1]}.
\end{gather}
Note that we have suppressed the dependence of $\up^{[j]}$ on $\wseq$ in the notation.

Recall that $\relongest$ denotes the longest element in the relative Coxeter groupoid $\reW$ associated with $\I\in \DiAIII_\pt$.  
The goal of this section is to establish \cref{thm:factorization}, which may be regarded as a super analogue of \cite[Conjecture~3.22]{DK19}. In the purely even case, this conjecture is proved in \cite[Theorem~8.1]{WZ23} for all finite types.
\begin{theorem}
\label{thm:factorization}
    (1) For any $w\in \reW$, the partial quasi $K$-matrix $\up_{\wseq}$ is independent of the choice of reduced expressions of $w$. (Hence it can be denoted by $\up_w$.)

    (2) The quasi $K$-matrix associated with $(\U(\I),\Ui(\I))$ admits a factorization $\up_\I=\up_{\relongest}$.
\end{theorem}

The rest of this subsection will be devoted to prove \cref{thm:factorization}. We first recall the following result.
\begin{proposition} {\rm \cite[Theorems~3.17 and 3.20]{DK19}}
      \label{prop:DK}
      \cref{thm:factorization} holds  if it holds for all its rank-two Satake subdiagrams.
\end{proposition}
We note that the proof of \cref{prop:DK} is largely formal once \cref{prop:tau0}, the super analogue of \cite[Proposition~3.18]{DK19}, is established.  
To prove \cref{prop:tau0}, we follow the strategy of both \cite[Proposition~3.18]{DK19} and \cite[Proposition~8.3]{WZ23}, where the braid group operator associated with the longest element of the Weyl group plays a central role. 

In our setting, recall that we have the Coxeter groupoid $W$ introduced in \cref{groupoiddef}.  
For each $\I\in \DiAIII_\pt$, we see that $w_\I(\alpha_{i,\I})=-\alpha_{\tau i,\I}$. Therefore, if we let $\tilde{\I}$ denote the element in $\DiAIII_\pt$ such that \[
|i,\I|=|i,\tilde{\I}|,\qquad d_{i,\I}=d_{\tau i,\tilde{\I}} ,\qquad \forall i\in \I,
\]
then the associated braid group operator $T_{w_\I}$ defines an isomorphism \[
T_{w_\I}: \U(\I) \to \U(\tilde{\I})
\] 
Moreover, we have the identities (for all $j\in \I$)
\begin{equation}
\label{Twi}
    \begin{gathered}
    T_{w_\I}(F_{j,\I})=-K_{\tau j,\tilde{\I}}^{-1} E_{\tau j,\tilde{\I}},\quad T_{w_\I}(E_{j,\I})=-F_{\tau j,\tilde{\I}} K_{\tau j,\tilde{\I}}, \quad T_{w_\I}(K_{j,\I})=(-1)^{|j,\tilde{\I}|}K_{\tau j,\tilde{\I}}^{-1},\\
     T_{w_\I}^{-1}(F_{j,\I})=-(-1)^{|j,\tilde{\I}|}E_{\tau j,\tilde{\I}}K_{\tau j,\tilde{\I}} ,\quad T_{w_\I}^{-1}(E_{j,\I})=-(-1)^{|j,\tilde{\I}|}K_{\tau j,\tilde{\I}}^{-1}F_{\tau j,\tilde{\I}} , \quad T_{w_\I}^{-1}(K_{j,\I})=(-1)^{|j,\tilde{\I}|}K_{\tau j,\tilde{\I}}^{-1}.
\end{gathered}
\end{equation}
The identities in \cref{Twi} are well known in the purely even case, and in our setting they can be derived by the same argument as in \cite[Lemma~3.4]{Ko14}. 
\begin{proposition}
   \label{prop:tau0}
 Let $\relongest =\bs_{i_1}\bs_{i_2}\cdots \bs_{i_m}$ be a reduced expression of $\relongest$.  Then we have
 \[
 T_{\bs_{i_1}} T_{\bs_{i_2}} \cdots T_{\bs_{i_{m-1}}}(\up_{i_m,\I})
 =\up_{i_m,\I}.
 \]
\end{proposition}

\begin{proof}
    We note that $w_\I=w_\bu \relongest$ by our constructions. 
    By definition of $\up_{i_m,\I}$ in \cref{Kmatrix}, we have $\up_{i_m,\I}=\up_{\tau i_m,\I}$. Let $w_{\bu,i_m}$ denote the longest element associated with $\I_{\bu,i_m}$ defined in \cref{Iib}. Hence $\relongest (\I)=\I$ and
    \begin{align*}
        T_{\bs_{i_1}} T_{\bs_{i_2}} \cdots T_{\bs_{i_{m-1}}}(\up_{i_m,\I})
        {=} T_{\relongest} T^{-1}_{\bs_{i_m}} (\up_{i_m,\I})=T_{w_\I}T_{w_\bu}^{-1}T_{w_\bu} T^{-1}_{w_{\bu,i_m}}(\up_{i_m,\I})
        \overset{\cref{Twi}}{=}\up_{i_m,\I}.
    \end{align*}
    This concludes the proof.
\end{proof}

\begin{proof}[Proof of \cref{thm:factorization}]
    By \cref{prop:tau0} and \cref{prop:DK}, it suffices to consider the real rank-two case where $\I_{\circ}=\{i,j,\tau i,\tau j\}$. In that case, the first statement in \cref{thm:factorization} is nontrivial only when $w=\relongest$, the longest element in $\reW$. Moreover, by the uniqueness in \cref{Kmatrix}, to prove the the first statement in the real rank two case, it suffices to show that 
    \begin{gather}
    \label{aqua}
        B_{\tau i,\I}\up_{\underline{\relongest}}=\up_{\underline{\relongest}} \sigma \tau (B_{i,\I}),\quad  B_{\tau j,\I}\up_{\underline{\relongest}}=\up_{\underline{\relongest}} \sigma\tau (B_{j,\I}),
    \end{gather}
    for any reduced expression $\underline{\relongest}$ of $\relongest$. As in the proof of \cref{thm:property}, there are three possibilities of real rank $2$ sub-diagrams. Then the verification of \cref{aqua} follows from a case-by-case checking as in \cite[Section 8.3]{WZ23} with the help of \cref{thm:property}. This finishes the proof.
\end{proof}

\subsection{Braid group actions on $\Omega$}
\label{sec:braid}
Let $\Br(\reW)$ (resp. $\Br(W_\bu)$) be the braid groupoid (resp. braid group) associated with $\reW$ (resp. $W_\bu$).  
In this subsection, we prove that the operators $\TT_i,\TT_i'$ $(i\in \I_\circ)$ satisfy the braid relations of $\Br(\reW)$.  
As a consequence, they induce an action of $\Br(W_\bu)\rtimes \Br(\reW)$ on the family
\[
\Omega:=\{\Ui(\I)\mid \I\in \DiAIII_\pt\}.
\]

Recall that, for any $\I\in \DiAIII_\pt$, there are three possible types of real rank \(2\) subdiagrams:
\begin{itemize}
\item if \(i\nsim j\), then \(\bs_i\bs_{j,\I}=\bs_j\bs_{i,\I}\);
\item if \(i\sim j\) and \(\{i,j\}\neq \{\pt,\pt+1\}\), then \(\bs_i\bs_j\bs_{i,\I}=\bs_j\bs_i\bs_{j,\I}\);
\item if \(i\sim j\) and \(\{i,j\}=\{\pt,\pt+1\}\), then \(\bs_i\bs_j\bs_i\bs_{j,\I}=\bs_j\bs_i\bs_j\bs_{i,\I}\).
\end{itemize}
For convenience, we write
\begin{align*}
\underbrace{\bs_i\bs_j\bs_i\cdots}_{k_{ij}}
=
\underbrace{\bs_j\bs_i\bs_j\cdots}_{k_{ij}}
\end{align*}
for the corresponding relation, where \(k_{ij}=2,3,4\) in the three cases above, respectively.

\begin{theorem}
  \label{thm:braidreW}
For $i\neq j\in \I_\circ$ and $i\neq \tau j$, we have
\begin{gather}
\label{eq:braidreW}
    \underbrace{\TT_{i} \TT_{j} \TT_{i}  \cdots }_{k_{ij}}
=\underbrace{\TT_{j} \TT_{i} \TT_{j}  \cdots}_{k_{ij}},\qquad 
\underbrace{\TT'_{i} \TT'_{j} \TT'_{i}  \cdots }_{k_{ij}}
=\underbrace{\TT'_{j} \TT'_{i} \TT'_{j}  \cdots}_{k_{ij}},
\end{gather}
\end{theorem}

\begin{proof}
 By \cref{thm:ibraid}(3), it suffices to prove the second equality in \cref{eq:braidreW}.  
The case \(i\nsim j\) is immediate.  
We treat the case \(i\sim j\) and \(\{i,j\}=\{\pt,\pt+1\}\); the remaining case is similar. In this case, we have $k_{ij}=4$. For any $x\in \Ui(\I)$, we have
\begin{equation}
\label{k=4}
    \begin{aligned}
        &\TT'_{i} \TT'_{j} \TT'_{i} \TT'_{j,\I}(x)\up_{i,\I} \cdot T'_{\bs_i}(\up_{j,\bs_j\bs_i\bs_j(\I)})\cdot T'_{\bs_i}T'_{\bs_j}(\up_{i,\bs_i\bs_j(\I)}) \cdot T'_{\bs_i}T'_{\bs_j}T'_{\bs_i}(\up_{j,r_j(\I)})\\
          \overset{\cref{eq:intertw}}{=}&
           \up_{i,\I}T'_{\bs_i}(\up_{j,\bs_j\bs_i\bs_j(\I)})T'_{\bs_i}T'_{\bs_j}(\up_i,\bs_i\bs_j(\I))T'_{\bs_i}T'_{\bs_j}T'_{\bs_i}(\up_{j,r_j(\I)})
            T'_{\bs_i}T'_{\bs_j}T'_{\bs_i}T'_{\bs_j,\I}(x)
    \end{aligned} 
\end{equation}
and 
\begin{equation}
\label{k=4'}
    \begin{aligned}
        &\TT'_{j} \TT'_{i} \TT'_{j} \TT'_{i,\I}(x)\up_{j,\I} \cdot T'_{\bs_j}(\up_{i,\bs_i\bs_j\bs_i(\I)})\cdot T'_{\bs_j}T'_{\bs_i}(\up_{j,\bs_j\bs_i(\I)}) \cdot T'_{\bs_j}T'_{\bs_i}T'_{\bs_j}(\up_{i,r_i(\I)})\\
          \overset{\cref{eq:intertw}}{=}&
           \up_{j,\I}T'_{\bs_j}(\up_{i,\bs_i\bs_j\bs_i(\I)})T'_{\bs_j}T'_{\bs_i}(\up_j,\bs_j\bs_i(\I))T'_{\bs_j}T'_{\bs_i}T'_{\bs_j}(\up_{i,r_i(\I)})
            T'_{\bs_j}T'_{\bs_i}T'_{\bs_j}T'_{\bs_i,\I}(x).
    \end{aligned} 
\end{equation}
Let $\bbw=\bs_i\bs_j \bs_i \bs_{j,\I}=\bs_j \bs_i\bs_j \bs_{i,\I}$. Since $T'_i$ satisfy the braid group relations, we have
\[
 T'_{\bs_i}T'_{\bs_j}T'_{\bs_i}T'_{\bs_j,\I}(x)=T'_{\bs_j}T'_{\bs_i}T'_{\bs_j}T'_{\bs_i,\I}(x).
\]
Comparing \cref{k=4,k=4'}, we see that, to complete the proof, it remains to show that
\begin{multline}
\label{fac}
\up_{i,\I}T'_{\bs_i}(\up_{j,\bs_j\bs_i\bs_j(\I)})T'_{\bs_i}T'_{\bs_j}(\up_{i,\bs_i\bs_j(\I)})T'_{\bs_i}T'_{\bs_j}T'_{\bs_i}(\up_{j,r_j(\I)})
    \\
=
\up_{j,\I}T'_{\bs_j}(\up_{i,\bs_i\bs_j\bs_i(\I)})T'_{\bs_j}T'_{\bs_i}(\up_{j,\bs_j\bs_i(\I)})T'_{\bs_j}T'_{\bs_i}T'_{\bs_j}(\up_{i,r_i(\I)}).
\end{multline}
By \cref{eq:sTs,upj,upwseq}, the two sides of \cref{fac} are precisely the factorizations of $\sigma(\up_{\bbw})$ corresponding to the two different reduced expressions.  
Therefore, \cref{thm:factorization} implies that the two sides of \cref{fac} are equal.
\end{proof}

For $w\in \reW$, take a reduced expression $w=\bs_{i_1}\bs_{i_2}\cdots\bs_{i_k,\I}$ and define
\begin{align}\label{def:tTT3}
\TT_{w}:= \TT_{i_1} \TT_{i_2} \cdots \TT_{i_k},
\qquad
\TT'_{w}:= \TT'_{i_1} \TT'_{i_2} \cdots \TT'_{i_k}.
\end{align}
By \cref{thm:braidreW}, these are now independent of the choice of reduced expressions for $w$.

\begin{theorem}
  \label{thm:braidW}
(1) There exists an action of $\Br(W_\bu) \rtimes \Br(\reW)$ on $\Omega$ as isomorphisms of algebras generated by $T'_{j} \; (j\in \I_\bu)$ and $\TT'_{i,\I} \;(i\in \I_\circ)$ for all $\I\in\DiAIII_a$.

(2) There exists an action of $\Br(W_\bu) \rtimes \Br(\reW)$ on $\Omega$ as isomorphisms of algebras generated by $T_{j} \; (j\in \I_\bu)$ and $\TT_{i,\I} \;(i\in \I_\circ)$ for all $\I\in\DiAIII_a$.
\end{theorem}

\begin{proof}
    It suffices to prove the first statement. The defining relations of $\Br(W_\bu) \rtimes \Br(\reW) $ consist of braid relations for $\Br(W_\bu)$, the braid relations for $\Br(\reW)$, and the following relation:
     \begin{align}
     \label{rbgoidrel}
         \bs_{\pm \pt} s_{j,\I}=s_j\bs_{\pm \pt,\I},\qquad \text{for any }j\in \I_\bu, \I\in\DiAIII_\pt
     \end{align}
     Let $i=\pm \pt$ and $\J=\bs_i(\I)$. Since $\ell(\bs_{i} s_j) = \ell(\bs_i) +1$, we have $T'_{\bs_i}T'_{j,\I}=T'_{j}T'_{\bs_i,\I}$.
     By \cref{lem:blackfix}, $\up_{i,\I}$ is fixed by $T_j'$. Hence we obtain that
     \begin{align}
T'_{j,\J}\TT'_{i,\I} (x) \up_{i,\J}\overset{\cref{eq:intertw}}{=}\up_{i,\J} T'_{j}T'_{\bs_i,\I} (x) =\up_{i,\J}T'_{\bs_i}T'_{j,\I}(x)\overset{\cref{eq:intertw}}{=}\TT'_{i,\I}T'_{j,\I}(x) \up_{i,\J}
\end{align}
for any $x\in \Ui(\I)$, hence we conclude that 
\[
\TT'_{i,\I}T'_{j,\I}=T'_{j,\J}\TT'_{i,\I}
\]
as desired. On the other hand, the braid relations for $\TT'_{k,\I},k\in \I_\circ$ are verified in \cref{thm:braidreW}. This concludes the proof.
\end{proof}

\begin{remark}[Compatible actions on modules]
Fix \(\I\in \DiAIII_\pt\).  For \(i\in \Ieven\cap \I_\circ\), the simple reflection \(\bs_i\) fixes \(\I\).  Thus we suppress the subscript \(\I\) in what follows.  Let \(M\) be an integral \(\U(\I)\)-module as in \cite[Definition 2.2]{BKK00}. By Lusztig's construction \cite{Lubook}, the braid group operators \(T'_i\) and \(T_i\) act on \(M\), which satisfy the compatibility relations
\[
T'_{i}(uv)=T'_{i}(u)T'_{i}(v),
\qquad
T_{i}(uv)=T_{i}(u)T_{i}(v),
\]
for \(u\in \U\) and \(v\in M\); see \cite[39.4.3]{Lubook}.

Since \(\Ui=\Ui(\I)\) is a subalgebra of \(\U\), the module \(M\) is also a \(\Ui\)-module by restriction.  For \(i\in \Ieven\cap \I_\circ\), define linear operators on \(M\) by
\[
\TT'_{i}(v):=\up_i T'_{\bs_i}(v),
\qquad
\TT_i(v):=T_{\bs_i}(\up_i^{-1})T_{\bs_i}(v).
\]
Then these operators are compatible with the relative braid group symmetries on \(\Ui\).  Namely, for every \(x\in \Ui\) and \(v\in M\), one can use \cref{thm:ibraid} to show that
\[
\TT'_{i}(xv)=\TT'_{i}(x)\TT'_{i}(v),
\qquad
\TT_{i}(xv)=\TT_{i}(x)\TT_{i}(v).
\]
Moreover, if \(i,j\in \Ieven\cap \I_\circ\) with \(i\neq j,\tau j\), then the operators \(\TT'_i\) and \(\TT_i\) satisfy the relative braid relations on \(M\).
This follows from the same argument as in the proof of \cref{thm:braidreW}, using the canonical factorization of the quasi \(K\)-matrix.  Thus the relative braid group symmetries associated with even simple roots admit compatible actions on integral \(\U\)-modules, in direct analogy with Lusztig's braid group action on integrable modules over quantum groups.

When the underlying Satake diagram is quasi-split (i.e., $\pt=0$ in \eqref{AIIIdiagram}), the formulas on the module level in \cite[Theorem C]{WZ25} still valid for $\TT'_i,\TT_i$ associated with $i\in \Ieven$ in the super setting. On the other hand, it would be interesting to study how the relative braid group symmetries associated with odd simple roots act on representations of \(\Ui\) attached to different Satake diagrams, as a super generalization of \cite{WZ25}.  
\end{remark}

\printbibliography

\end{document}